\documentclass[final,12pt,a4paper]{amsart}

\usepackage{graphicx}

\usepackage[all]{xy}
\newcommand{\textxy}[1]{\text{\emph{#1}}}

\usepackage{placeins}
\usepackage{enumitem}
\usepackage{amssymb}
\usepackage{latexsym}
\usepackage{amsmath}
\usepackage{mathrsfs}
\usepackage{array,booktabs}

\usepackage[usenames,dvipsnames]{xcolor}
\usepackage{hyperref}
\hypersetup{
    colorlinks,
    linkcolor={red!50!black},
    citecolor={blue!50!black},
    urlcolor={blue!80!black} }  

\newcommand{\harxiv}[1]{\href{http://arxiv.org/abs/#1}{\texttt{arXiv:#1}}}
\newcommand{\hyref}[2]{\hyperref[#2]{#1~\ref*{#2}}}

\usepackage{fullpage}

\newcommand{\listskip}{\\[0.7ex]}

\theoremstyle{plain}
\newtheorem{theorem}{Theorem}[section]

\newtheorem{lemma}[theorem]{Lemma}
\newtheorem{corollary}[theorem]{Corollary}
\newtheorem{proposition}[theorem]{Proposition}

\theoremstyle{definition}
\newtheorem{remark}[theorem]{Remark}
\newtheorem{example}[theorem]{Example}
\newtheorem{definition}[theorem]{Definition}
\newtheorem*{notation}{Notation}

\newcommand{\IC}{{\mathbb{C}}}
\newcommand{\IH}{{\mathbb{H}}}
\newcommand{\IP}{{\mathbb{P}}}
\newcommand{\IR}{{\mathbb{R}}}
\newcommand{\IZ}{{\mathbb{Z}}}
\newcommand{\IN}{{\mathbb{N}}}

\newcommand{\kk}{{\mathbf{k}}}

\newcommand{\sA}{\mathsf{A}}
\newcommand{\sB}{\mathsf{B}}
\newcommand{\sC}{\mathsf{C}}
\newcommand{\sD}{\mathsf{D}}

\newcommand{\sH}{\mathsf{H}}

\newcommand{\sK}{\mathsf{K}}
\newcommand{\sL}{\mathsf{L}}
\newcommand{\sM}{\mathsf{M}}
\newcommand{\sN}{\mathsf{N}}

\newcommand{\sR}{\mathsf{R}}
\newcommand{\sS}{\mathsf{S}}

\newcommand{\sU}{\mathsf{U}}
\newcommand{\sV}{\mathsf{V}}
\newcommand{\sW}{\mathsf{W}}
\newcommand{\sX}{\mathsf{X}}
\newcommand{\sY}{\mathsf{Y}}
\newcommand{\sZ}{\mathsf{Z}}

\DeclareMathAlphabet{\mathpzc}{OT1}{pzc}{m}{it}

\newcommand{\cF}{\mathcal{F}}

\newcommand{\cK}{\mathcal{K}}

\newcommand{\cS}{\mathcal{S}}
\newcommand{\cT}{\mathcal{T}}

\newcommand{\cX}{\mathcal{X}}
\newcommand{\cY}{\mathcal{Y}}
\newcommand{\cZ}{\mathcal{Z}}

\newcommand{\LLambda}{\Lambda(r,n,m)}

\newcommand{\Db}{\sD^b}
\newcommand{\Kb}{\sK^b}
\newcommand{\Ksplit}{K^{\mathrm{split}}_0}

\newcommand{\coloneqq}{\mathrel{\mathop:}=}

\newcommand{\rk}[1]{\mathrm{rk}\, #1}

\DeclareMathOperator{\image}{\mathsf{Im}}

\newcommand{\orth}{^\perp}
\newcommand{\inv}{^{-1}}
\newcommand{\blank}{-}

\DeclareMathOperator{\susp}{\mathsf{susp}}
\DeclareMathOperator{\cosusp}{\mathsf{cosusp}}

\DeclareMathOperator{\id}{{\mathsf{id}}}
\newcommand{\ind}[1]{\mathsf{ind}(#1)}
\renewcommand{\mod}[1]{\mathsf{mod}(#1)}
\newcommand{\proj}[1]{\mathsf{proj}(#1)}
\DeclareMathOperator{\add}{\mathsf{add}}

\newcommand{\lmu}[1]{\sL_{#1}}
\newcommand{\rmu}[1]{\sR_{#1}}

\newcommand{\thick}[2]{\mathsf{thick}_{#1}(#2)}

\newcommand{\siltsub}[1]{\mathsf{siltcat}(#1)}        
\newcommand{\silt}[1]{\mathsf{silt}_2(#1)}

\newcommand{\stab}[1]{\mathsf{stab}(#1)}
\newcommand{\charge}{Z}
\newcommand{\fl}{{\star}}  

\newcommand{\simples}{{\mathcal S}}

\DeclareMathOperator{\Hom}{\mathrm{Hom}}

\newcommand{\posetsilt}{\IP_1}
\newcommand{\posetpair}{\IP_2}

\renewcommand{\setminus}{\backslash}

\newcommand{\clext}[1]{{\langle #1\rangle}}    

\newcommand{\tri}[3]{#1\rightarrow #2\rightarrow #3\rightarrow \Sigma #1}

\newcommand{\isom}{ \text{{\hspace{0.48em}\raisebox{0.8ex}{${\scriptscriptstyle\sim}$}}}
                    \hspace{-0.65em}{\rightarrow}\hspace{0.3em}}

\newcommand{\embed}{\hookrightarrow}

\newcommand{\too}{\longrightarrow}
\renewcommand{\iff}{\Longleftrightarrow}
\newcommand{\iffdef}{\underset{\text{\tiny def}}{\Longleftrightarrow}}
\newcommand{\rightlabel}[1]{\stackrel{#1}{\longrightarrow}}

\renewcommand{\phi}{\varphi}
\renewcommand{\epsilon}{\varepsilon}

\newcommand{\bib}[6]{{\bibitem{#2} #3: {\emph{#4},} #5#6.}}

\newcommand{\bibno}[1]{}

\usepackage{tikz}
\usetikzlibrary{decorations.pathmorphing}
\usetikzlibrary{decorations.pathreplacing}
\usetikzlibrary{positioning}
\usetikzlibrary{calc}
\usetikzlibrary{backgrounds}
\usetikzlibrary{shapes.misc}

\usepackage{pifont} 

\begin{document}

\title[Discrete derived categories II]{Discrete derived categories II \\
                       \scalebox{0.86}{The silting pairs CW complex and the stability manifold}}

\author{Nathan Broomhead}
\address{Faculty of Mathematics,
         Bielefeld University,
         PO Box 100 131,
         33501 Bielefeld, Germany}
\email{nbroomhe@math.uni-bielefeld.de}

\author{David Pauksztello}
\address{School of Mathematics,
        The University of Manchester,
        Oxford Road,
        Manchester,  M13 9PL, United Kingdom}
\email{david.pauksztello@manchester.ac.uk}

\author{David Ploog}
\address{Institut f\"{u}r Algebraische Geometrie,
         Fa\-kul\-t\"{a}t f\"{u}r Ma\-the\-ma\-tik und Physik,
         Leibniz Universit\"{a}t Hannover,
         Welfengarten 1, 30167 Hannover, Germany}
\email{ploog@math.uni-hannover.de}

\begin{abstract}
Discrete derived categories were studied initially by Vossieck \cite{Vossieck} and later by Bobi\'nski, Gei\ss, Skowro\'nski \cite{BGS}. In this article, we define the CW complex of silting pairs for a triangulated category and show that it is contractible in the case of discrete derived categories. We provide an explicit embedding from the silting CW complex into the stability manifold. By work of Qiu and Woolf \cite{QW}, there is a deformation retract of the stability manifold onto the silting pairs CW complex. We obtain that the space of stability conditions of discrete derived categories is contractible.
\end{abstract}

\maketitle

{\small
\setcounter{tocdepth}{1}
\tableofcontents
}

\addtocontents{toc}{\protect{\setcounter{tocdepth}{-1}}}  

\section*{Introduction}
\addtocontents{toc}{\protect{\setcounter{tocdepth}{1}}}   

\noindent
The class of algebras whose bounded derived categories are discrete, the so-called \emph{derived-discrete algebras}, was introduced by Vossieck in \cite{Vossieck}. This family constitutes a particularly interesting source of examples:
\begin{itemize}
\item they are intermediate in complexity between finite and tame representation type hereditary algebras;
\item they can have arbitrarily large global dimension;
\item the Auslander--Reiten quivers of their bounded derived categories are completely described \cite{BGS}.
\end{itemize}
Derived-discrete algebras provide a 3-parameter family of triangulated categories in representation theory, allowing explicit calculations in arbitrary global dimension.
Similarly concrete and detailed knowledge of the AR quivers of the bounded derived categories is available for hereditary algebras; however, these satisfy peculiarly strong homological properties which are unavailable in general.
In this paper, we shall examine the important interplay between some geometric aspects of finite-dimensional algebras with some combinatorial aspects.

\medskip

\noindent{\it Geometric aspects: Bridgeland stability conditions.}
Stability conditions were introduced by Bridgeland in \cite{Bridgeland} as a new invariant of triangulated categories inspired by work in mathematical physics. Stability conditions may be considered as a `continuous' generalisation of bounded t-structures, which are central to the study of triangulated categories in the following ways:
\begin{itemize}
\item they construct abelian categories, \emph{hearts}, inside triangulated categories;
\item they construct different cohomology theories on triangulated categories.
\end{itemize}
 Combining this cohomological information can be achieved by passing to Bridgeland's stability manifold \cite{Bridgeland}, which is a `moduli space' of stability conditions encoding the information of (most of) the bounded t-structures in a triangulated category.

Unfortunately, computations with stability conditions and stability manifolds, particularly in geometric settings, are notoriously difficult. This has led some people to seek to determine and understand the stability manifolds of finite-dimensional algebras, for example in the work of Dimitrov, Haiden, Katzarkov and Kontsevich \cite{Kontsevich}, which while still difficult, may at least be manageable. As a concrete example: amongst experts there is a feeling that, when non-empty, the stability manifold is contractible. Thus far, for algebraic examples, this is known explicitly only for the bounded derived categories of the algebras $\kk A_2, \kk\tilde A_1, \kk\tilde A_2, \Lambda(1,2,0)$ (we refer to Section~\ref{sec:discrete} for the precise definition); see \cite{BQS,Okada,DK,Woolf}, respectively. In Theorem~\ref{thm:stab-contractible}, we show the contractibility of the stability manifold for the entire family of finite global dimension derived-discrete algebras.

\medskip

\noindent{\it Combinatorial aspects: silting objects.}
Introduced in \cite{Keller-Vossieck}, silting objects are a generalisation of tilting objects in which we no longer insist that negative self-extensions vanish. They sit on the cusp of classical tilting theory and cluster-tilting theory.

Inspired by the combinatorics of classical tilting theory \cite{Riedtmann-Schofield, Unger}, Aihara and Iyama introduced a partial order on silting objects in \cite{AI}. For hereditary algebras, restricting this to a partial order on so-called `two-term silting objects' recovers the corresponding exchange graph of cluster-tilting objects, placing the combinatorics of silting objects at the centre of cluster-tilting theory. For an excellent survey of these connections see \cite{Bruestle-Yang}.

In Section~\ref{sec:silting-pairs}, we introduce a new poset: the poset of \emph{silting pairs}. This poset and its associated topological space provide new invariants for triangulated categories. We show that this poset satisfies certain good finiteness properties in the case of derived-discrete algebras, making it a CW poset and the space a regular CW complex, the \emph{silting pairs CW complex} (Theorem~\ref{thm:CW-poset} and Proposition~\ref{prop:finiteness}). We show that the partial order on silting pairs is very closely related to the silting analogue of the classical tilting concept of `Bongartz completion' in Proposition~\ref{prop:po-equivalent}. In Theorem~\ref{thm:CW-contractible}, we show that the silting pairs CW complex of a derived-discrete algebra is contractible, and in particular, that its silting quiver in the sense of \cite{AI} is connected.

As a consequence, we obtain in Corollary~\ref{cor:contractible} that the CW complex of two-term silting objects for discrete derived catgeories is also contractible. This may be considered as a result on cluster structures in higher global dimension. This suggests that two-term silting objects for derived-discrete algebras warrant further investigation; see Remark~\ref{rem:cluster-connection}.

\medskip

\noindent{\it The connection.}
Let $\Lambda$ be a finite-dimensional algebra of finite global dimension. Whereas tilting objects in $\Db(\Lambda)$ detect the finite-dimensional algebras which are derived-equivalent to $\Lambda$, silting objects in $\Db(\Lambda)$ determine the finite-dimensional algebras which have module categories sitting inside $\Db(\Lambda)$ as hearts. In other words, silting objects determine the bounded t-structures in $\Db(\Lambda)$ whose hearts are `algebraic'. In the case of derived-discrete algebras, results of the first article \cite{discrete-one} show that all hearts inside a discrete derived category are algebraic. In particular, this means that the silting pairs CW complex captures essentially the same information as the stability manifold. In fact, the silting pairs CW complex is a deformation retraction of the stability manifold.

At this point, we would like to stress the following aspect of our philosophy. The stability manifold is an interesting but often difficult geometric invariant. The poset of silting pairs, and the induced CW complex, give a concrete combinatorial and representation-theoretic interpretation and viewpoint for `algebraic' stability conditions. In the case of discrete derived categories, the silting pairs CW complex is a half-dimensional analogue of the stability manifold which contains the same information.

\medskip

\noindent{\it Relationship with the work of Qiu and Woolf.}
At a late stage of writing, we became aware of the independent and concurrent work of Qiu and Woolf \cite{QW}. From the point of view of tilting t-structures at torsion pairs, they define a poset $\mathrm{Int}(\sD)$, which, after formally adjoining a bottom element, is isomorphic to our poset of silting pairs. Both strategies entail obtaining a deformation retraction of the stability manifold onto the classifying space of this poset and showing that this classifying space is contractible.

However, there are important differences in exposition and content. Our approach is more algebraic, and provides important connections to representation-theoretic constructions such as Bongartz completion. We develop the theory of silting pairs CW complexes as an algebraic counterpart to stability manifolds. We are able to apply a detailed understanding of the representation theory of derived-discrete algebras to show that the stability manifold is connected. Combining this result with the approach taken in \cite{QW}, where each component of the stability manifold of a derived-discrete algebra is shown to be contractible, one gets that the whole stability manifold is contractible.

In this article, we prove the contractibility of the CW complex and give the embedding into the stability space but refrain from proving that the former is a deformation retract of the latter. For that, we cite \cite{QW}. We have an independent proof which is different from that of \cite{QW}; the techniques used there may warrant another article on the subject.

\subsection*{Acknowledgments}
We are grateful to Aslak Bakke Buan, Martin Kalck, Henning Krause, and Dong Yang.
We would like to thank Yu Qiu and Jon Woolf for kindly sharing their preprint \cite{QW}.
We are grateful to Osamu Iyama for pointing out the reference \cite{Aihara} and to Jairui Fei for pointing out \cite{DF}.
We would like to thank the referee for a careful reading and valuable comments.
The second author acknowledges the financial support of the EPSRC of the United Kingdom through the grant EP/K022490/1.


\section{Preliminaries}

\noindent
In this section we collect some notation, mostly standard. We work over an algebraically closed field $\kk$. The suspension functors (otherwise known as shift or translation) of all triangulated categories are denoted by $\Sigma$. All categories and functors are supposed to be $\kk$-linear. Subcategories are supposed to be full, additive and closed under isomorphisms. Objects will always be considered up to isomorphisms, and by abuse of terminology we shall identify objects with their isomorphism classes. All triangles we mention are distinguished, and all functors between triangulated categories are triangle functors.

\subsection{General categorical notions}
Let $\sD$ be a $\kk$-linear triangulated category. In this article, we assume $\sD$ to be a Hom-finite, Krull-Schmidt category (unless we explicitly work in greater generality), i.e.\ every object has a unique decomposition as a direct sum of finitely many indecomposable objects.
For two objects $A,B$ of $\sD$, we use the traditional shortcut notation $\Hom^i(A,B) = \Hom(A,\Sigma^i B)$ resembling Ext spaces in abelian categories. We write
\[ \Hom^{>0}(A,B) = \bigoplus_{i>0} \Hom(A,\Sigma^i B) \]
for aggregated homomorphism spaces, and similarly for obvious variants.

Occasionally, we will use the notation $\Hom(\sA,\sB)$ for subcategories $\sA,\sB\subseteq\sD$ to mean the collection of all morphisms $A\to B$, where $A\in\sA, B\in\sB$.

\subsection{Subcategory constructions} \label{sub:subcategory-constructions}
For a full subcategory $\sC$ of $\sD$, we write $\ind{\sC}$ for the set of \emph{indecomposable} objects of $\sC$ up to isomorphism, and consider the following subcategories associated with $\sC$:

\medskip
\noindent
\begin{tabular}{@{} p{0.13\textwidth} @{} p{0.87\textwidth} @{}}
$\sC\orth$,      & the \emph{right orthogonal} to $\sC$, the full subcategory of $D\in\sD$ with $\Hom(\sC,D)=0$, \\
${}\orth\sC$ ,   & the \emph{left orthogonal} to $\sC$, the full subcategory of $D\in\sD$ with $\Hom(D,\sC)=0$. \\
                 & If $\sC$ is closed under suspensions and cosuspensions, then $\sC\orth$ and $\orth\sC$ are
                  triangulated subcategories of $\sD$. \listskip
$\thick{}{\sC}$, & the \emph{thick subcategory generated by $\sC$}, the smallest thick (i.e.\ triangulated and closed under direct summands) subcategory of $\sD$ containing $\sC$. \listskip
$\susp(\sC)$
and
$\cosusp(\sC)$,  & the \emph{(co-)suspended subcategory generated by $\sC$}, the smallest full subcategory of
                   $\sD$ containing $\sC$ which is closed under (co-)suspension, extensions and taking direct
                   summands. \listskip
$\add(\sC)$,     & the \emph{additive subcategory} of $\sD$ containing $\sC$, the smallest full subcategory of
                   $\sD$ containing $\sC$ which is closed under finite coproducts and direct summands. \listskip
$\clext{\sC}$,   & the smallest full subcategory of $\sD$ containing $\sC$ that is closed under extensions
                   and direct summands. \\
\end{tabular}

\medskip\noindent
For two full subcategories $\sC_1,\sC_2$ of $\sD$ we denote by $\sC_1 * \sC_2$ the full subcategory of all objects $D$ occurring in triangles $C_1\to D\to C_2\to\Sigma C_1$ with $C_1\in\sC_1$ and $C_2\in\sC_2$. This construction can be iterated and is associative by the octahedral axiom, so that we will write $\sC_1 * \sC_2 * \sC_3$ etc.\ for more than two factors.

A subcategory $\sC$ is \emph{extension-closed} if $\sC*\sC=\sC$, or equivalently, $\sC*\sC\subseteq\sC$.

\subsection{Categorical approximations} \label{sub:approximations}
For the following definitions, let $\sD$ be an additive category and $\sC$ a full subcategory of $\sD$.

A \emph{right $\sC$-approximation} of an object $D\in\sD$ is a morphism $C\to D$ with $C\in\sC$ such that the induced maps $\Hom(C',C)\to\Hom(C',D)$ are surjective for all $C'\in\sC$.
A morphism $f\colon C\to D$ is called a \emph{minimal right $\sC$-approximation} if $fg=f$ is only possible for isomorphisms $g\colon C\to C$.
A \emph{(minimal) left $\sC$-approximation} is defined dually.

The subcategory $\sC$ is called \emph{functorially finite in $\sD$} if every object of $\sD$ has a right $\sC$-approximation and a left $\sC$-approximation.

\subsection{t-structures and co-t-structures} \label{sub:t-structures}

Following \cite{BBD}, a \emph{t-structure} is a pair of full subcategories $(\sX,\sY)$ such that $\sX*\sY=\sD$ and $\orth\sY=\sX, \sY=\sX\orth$ and $\Sigma \sX \subseteq \sX$. The last inclusion implies $\Sigma^{-1} \sY \subseteq \sY$. We will only consider \emph{bounded} t-structures, i.e.\ posit the requirement $\bigcup_{i\in \IZ} \Sigma^i \sX = \bigcup_{i \in \IZ} \Sigma^i \sY = \sD$.

If $(\sX,\sY)$ is a t-structure, then for any $D\in\sD$ there exist unique triangles $\tri{X}{D}{Y}$ with $X\in\sX$ and $Y\in\sY$; these triangles depend functorially on $D$, hence $X$ and $Y$ are called the \emph{right} and \emph{left truncation} of $D$, respectively. Another way of expressing the functoriality of the decomposition triangles is this: the inclusion $\sX\embed\sD$ has a right adjoint (given by $D\mapsto X$) and $\sY\embed\sD$ has a left adjoint. In particular, truncations are minimal approximations. Note that `t-structure' stands for `truncation structure'.

Furthermore, for a t-structure $(\sX,\sY)$, the intersection $\sX \cap \Sigma \sY$ is an abelian subcategory of $\sD$ called the \emph{heart} of $(\sX,\sY)$. It is possible to reconstruct a bounded t-structure $(\sX,\sY)$ from its heart $\sH$ by $\sX=\susp\sH$ and $\sY=\cosusp\Sigma\inv\sH$. Finally, $\sX$ is called the \emph{aisle} and $\sY$ the \emph{co-aisle} of the t-structure; always $\sX={}\orth\sY$ and $\sY=\sX\orth$.

By an \emph{algebraic t-structure} we mean a bounded t-structure whose heart is an abelian category of finite length, possessing only finitely many simple objects.

\medskip\noindent
A \emph{co-t-structure} is a pair of full subcategories $(\sX,\sY)$ with $\sX*\sY=\sD$ and $\orth\sY=\sX, \sY=\sX\orth$ and $\Sigma^{-1} \sX \subseteq \sX$; see \cite{Pauksztello}, or \cite{Bondarko} for the same notion with a different name. The notions of \emph{(co-)aisle} and \emph{bounded} are defined as for t-structures. However, for a co-t-structure $(\sX,\sY)$, the inclusion of the (co-)aisle does not necessarily possess an adjoint. Moreover, the \emph{co-heart} $\Sigma \sX \cap \sY$ is an additive subcategory of $\sD$ but not necessarily abelian.

\subsection{Silting objects and silting subcategories} \label{sub:siltings}
A full subcategory $\sM$ of $\sD$ is \emph{partial silting} if $\Hom^{>0}(\sM,\sM)=0$. It is called \emph{silting} if it is partial silting and $\thick{}{\sM}=\sD$. An object $M\in\sD$ is called a \emph{silting object} if $\add(M)$ is a silting subcategory. These notions are from \cite{Keller-Vossieck} and generalise tilting objects; our terminology follows \cite{AI}.

Two silting objects $M,M'\in\sD$ are \emph{equivalent} if and only if $\add(M) = \add(M')$.

It is easy to see that silting subcategories are extension-closed.
We collect the following facts from \cite{AI}.

\begin{proposition} \label{prop:silt-properties}
Let $\sD$ be a Krull-Schmidt triangulated category with a silting subcategory $\sM$. We have the following.
\begin{enumerate}[label=(\arabic*)]
\item If $\sD$ has a silting object then any silting subcategory is additively generated by a silting object. \label{item:silting-object}
\item $\sD$ has a silting object if and only if $K_0(\sD)$ is free of finite rank. \label{item:finite-rank}
\end{enumerate}
\end{proposition}

\begin{proof}
Statement~\ref{item:silting-object} is \cite[Proposition 2.20]{AI}. Statement~\ref{item:finite-rank} follows from \ref{item:silting-object} and \cite[Theorem 2.27]{AI}.
\end{proof}

Under some mild assumptions, Aihara and Iyama \cite{AI} provide a way to understand certain silting subcategories of $\sD$ by going to a quotient.
\begin{theorem}[Silting reduction {\cite[Theorem 2.37]{AI}}] \label{thm:silt-red}
Let $\sD$ be a Krull-Schmidt triangulated category, $\sU\subset\sD$ a thick, functorially finite subcategory and $F\colon\sD \to \sD/\sU$ the canonical functor. Then for any silting subcategory $\sN$ of $\sU$, there is a bijection
\[ \{ \text{silting subcategories $\sM$ of $\sD$} \mid \sN \subseteq \sM \} \isom
   \{ \text{silting subcategories of $\sD/\sU$} \}, \quad \sM \mapsto F(\sM) .
\]
\end{theorem}

\subsection{K\"onig--Yang correspondences} \label{sub:Koenig-Yang}
The following explains the connection between the various concepts defined above.

\begin{theorem}[\cite{Koenig-Yang} Theorems 6.1, 7.11, 7.12] \label{thm:Koenig-Yang}
\label{thm:koenig-yang}
Let $\Lambda$ be a finite-dimensional $\kk$-algebra. There are bijections between
\begin{enumerate}[label=(\arabic*)]
\item equivalence classes of silting objects in $\sK^b(\proj \Lambda)$,
\item algebraic t-structures in $\sD^b(\Lambda)$,
\item bounded co-t-structures in $\sK^b(\proj \Lambda)$.
\end{enumerate}
\end{theorem}

The algebras $\Lambda$ we have in mind in this article are of finite global dimension. Moreover, all their bounded t-structures are algebraic. Therefore, in these examples, the K\"onig--Yang correspondences provide bijections between silting subcategories, bounded t-structures and bounded co-t-structures in $\sD^b(\Lambda)$. These correspondences are crucial for our work. For readers unfamiliar with them, an explicit description with a concrete example is given in Appendix \ref{app:examples}.

\begin{remark} \label{rem:koenig-yang}
We point out that what we call K\"onig--Yang correspondences are in fact consequences of the work of many authors, including K\"onig and Yang. The paper \cite{Koenig-Yang} is the first place where all these correspondences are collected in the case of finite-dimensional algebras. We direct the reader to the introduction of \cite{Koenig-Yang} where an overview of the literature is given.
\end{remark}

\subsection{Abelian categories}
Let $\sH$ be an essentially small abelian category. We denote the set of simple objects of $\sH$ by $\simples(\sH)$. Given a set $\cS$ of objects of $\sH$, we denote by $\clext{\cS}_\sH$ its extension closure in $\sH$.

A \emph{torsion pair} $(\cT,\cF)$ in $\sH$ consists of two full subcategories $\cT,\cF\subseteq\sH$ such that $\Hom(\cT,\cF)=0$ and that $\cT,\cF$ are maximal with this property. Torsion pairs on hearts of bounded t-structures can be used to tilt the t-structure, see Section~\ref{sec:compatibility}.


\section{Background on discrete derived categories}
\label{sec:discrete}

\noindent
In this section, we recall some pertinent facts about discrete derived categories. For more details, the reader is advised to consult the articles \cite{BGS,discrete-one,Vossieck}.
Let $Q(r,n,m)$ be the quiver with $n+m$ vertices which make up an oriented cycle of length $n>0$ and a tail of length $m \geq 0$, as shown below. We define an admissible ideal $I(r,n,m)$ of the (infinite-dimensional) path algebra $\kk Q(r,n,m)$ generated by the $r >0$ length-two zero relations along the cycle of the quiver as indicated. We denote by $\LLambda$ the bound path algebra $\kk Q(r,n,m)/I(r,n,m)$.

\medskip
\begin{center}
\includegraphics[width=0.45\textwidth]{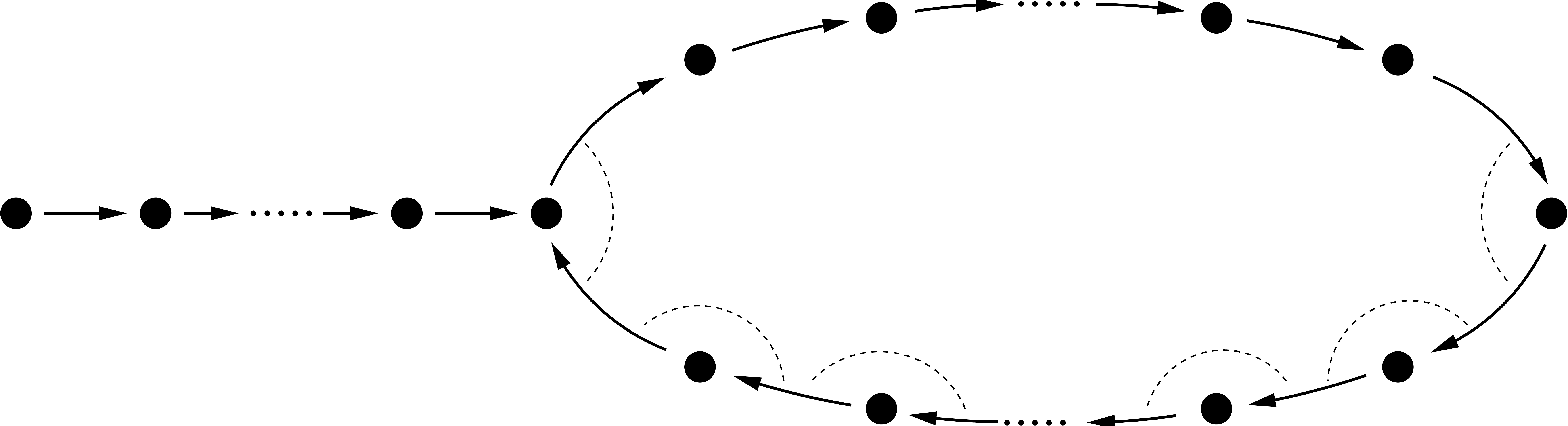}
\end{center}
\medskip

\noindent
Following \cite{Vossieck}, a derived category $\sD$ is \emph{discrete} if for every map $v\colon \IZ\to K_0(\sD)$ there are only finitely many isomorphism classes of objects $D\in\sD$ with $[H^i(D)]=v(i)\in K_0(\sD)$ for all $i\in\IZ$.

Algebras whose bounded derived categories are discrete were classified in \cite{Vossieck} and a derived Morita classification was given in \cite{BGS}. Such algebras are called \emph{derived-discrete algebras}: a finite-dimensional algebra is derived-discrete if and only if it is derived-equivalent to either a representation finite and hereditary algebra (i.e.\ the path algebra of a simply-laced Dynkin diagram) or the bound path algebra $\LLambda$ for some values of $r,n,m$.

\medskip\noindent {\it Convention: We always assume $n>r>0$.} \medskip

\noindent
Thus we are restricting our attention to those derived-discrete algebras which are of finite global dimension but not representation-finite.

The bounded derived categories of derived-discrete algebras also satisfy the following strong finiteness property, which is crucial for us: ultimately, it will allow us to understand the stability manifold completely in terms of silting pairs.

\begin{proposition}[{\cite[Proposition 6.1]{discrete-one}}] \label{prop:length}
 Any heart of a t-structure of a discrete derived category has only a finite number of indecomposable objects up to isomorphism, and is a length category.
\end{proposition}

As remarked in \cite{discrete-one}, this means that each heart inside a discrete derived category is the module category of a finite-dimensional algebra of finite representation type.

\subsection{The Auslander--Reiten quiver of a discrete derived category}

By work of Bobi\'nski, Gei\ss\ and Skowro\'nski in \cite[Theorem B]{BGS}, the Auslander--Reiten (AR) quiver of $\Db(\LLambda)$ has a very pleasant structure: It has precisely $3r$ components, of which $2r$ components are of type $\IZ A_\infty$, the so-called $\cX$ and $\cY$ components, and the remaining $r$ components are of type $\IZ A_\infty^\infty$: the $\cZ$ components $\cZ^0,\ldots,\cZ^{r-1}$. In this paper, only the behaviour of the $\cZ$ components will be relevant.

Indecomposable objects in $\cZ^k$ are labelled $Z^k_{ij}$ with $i,j \in \IZ$. The indecomposable object $Z^k_{ij}$ only admits irreducible morphisms to $Z^k_{i+1,j}$ and $Z^k_{i,j+1}$.
The action of the suspension functor on objects of $\cZ$ components, with $k = 0,\ldots,r-2$, is given by
\[
\Sigma Z^k_{ij} = Z^{k+1}_{ij}, \quad \text{and} \quad \Sigma Z^{r-1}_{ij} = Z^0_{i+r+m,j+r-n}.
\]

The homomorphisms between indecomposable objects were studied in detail in \cite{discrete-one}. Below we recall the Hom-hammocks from objects in the $\cZ$ components. The pictures in \cite{discrete-one} may be useful in understanding the structure of $\Db(\LLambda)$ in more detail.

\begin{proposition}[{\cite[Proposition 2.4]{discrete-one}}] \label{prop:Z-hammocks}
Let $A=Z^k_{ij}\in\ind{\cZ^k}$. For any indecomposable object $B\in\ind{\cZ}$ we have:
\[
\Hom(A,B) \neq 0 \iff  B = Z^k_{ab} \textrm{ for } a \geq i,\, b \geq j, \textrm{ or, } B = \Sigma Z^k_{ab} \textrm{ for } a \leq i-1,\, b \leq j-1.
\]
\end{proposition}

The following facts from \cite{discrete-one} will be useful:

\begin{proposition}[{\cite[Proposition 6.5]{discrete-one}}]\label{prop:embedding-A}
Let $Z \in \ind{\cZ}$ and $\sZ = \thick{\Db(\LLambda)}{Z}$. Then $\sZ$ is functorially finite in $\Db(\LLambda)$ and $\sZ\orth \simeq \Db(\kk A_{n+m-1})$.
\end{proposition}

\begin{lemma}[{\cite[Corollary 6.9]{discrete-one}}] \label{lem:Z-in-silting}
Any silting subcategory of $\Db(\LLambda)$ contains an indecomposable object from some $\cZ$ component.
\end{lemma}


\section{The poset of silting pairs}
\label{sec:silting-pairs}

\noindent
In this section we prove some technical results about silting subcategories. Let $\sD$ be an arbitrary Krull-Schmidt triangulated category.
We recall the definition of silting mutation from \cite[Definition 2.5]{Iyama-Yoshino}. We refer to Section~\ref{sub:subcategory-constructions} for the $*$ composition of subcategories and to Section~\ref{sub:approximations} for generalities on approximations.

\begin{definition}
A \emph{silting pair} $(\sM,\sM')$ consists of a silting subcategory $\sM\subset\sD$ and a functorially finite subcategory $\sM' \subseteq \sM$.
We call
\begin{align*}
\rmu{\sM'}(\sM) & \coloneqq (\Sigma\inv\sM * \sM') \cap (\Sigma\inv\sM')\orth && \text{the \emph{right mutation} of $\sM$ at $\sM'$, and} \\
\lmu{\sM'}(\sM) & \coloneqq (\sM' * \Sigma\sM) \cap {}\orth\Sigma\sM'         && \text{the \emph{left mutation} of $\sM$ at $\sM'$.}
\end{align*}
\end{definition}

More concrete descriptions of the mutated categories are \cite{Iyama-Yoshino} and \cite[Definition 2.30]{AI}:
\[ 
\begin{array}{l @{\quad\text{with}\quad} l}
 \rmu{\sM'}(\sM) = \add{\big( \sM' \cup \{\rmu{\sM'}(M) \mid M\in \sM\} \big) } & \Sigma\inv M \too \rmu{\sM'}(M') \too M' \xrightarrow{r_M} M , \\[0.5ex]
 \lmu{\sM'}(\sM) = \add{\big( \sM' \cup \{\lmu{\sM'}(M) \mid M\in \sM\} \big) } & M \xrightarrow{l_M} M' \too \lmu{\sM'}(M) \too \Sigma M
\end{array} 
\]
where we choose right and left $\sM'$-approximations $r_M$ and $l_M$ for each $M\in\sM$. The invariant definition shows that the subcategories $\rmu{\sM'}(\sM)$ or $\lmu{\sM'}(\sM)$ do not depend on the choice of approximations. The crucial property of this operation is that $\rmu{\sM'}(\sM)$ and $\lmu{\sM'}(\sM)$ are silting subcategories \cite[Theorem 2.31]{AI}. Right and left mutations are mutually inverse; see \cite[Theorem 2.33]{AI}. The extreme cases $\sM'=0$ and $\sM'=\sM$ are allowed and lead to $\rmu{0}(\sM)=\Sigma\inv \sM$ and $\rmu{\sM}(\sM)=\sM$.

For a silting subcategory $\sM$, we define $\Ksplit(\sM)$ to be the \emph{split Grothendieck group} of $\sM$, i.e.\ the free group on isomorphism classes of objects of $\sM$ modulo the relation $[M_1 \oplus M_2] = [M_1] + [M_2]$; see, for example, \cite{Bondarko}. By \cite[Theorem 5.3.1]{Bondarko} we have that $K_0(\sD) \cong \Ksplit(\sM)$ for any silting subcategory $\sM$ of $\sD$.

\begin{definition} \label{def:irreducible-mutation}
Let $\sM$ be a silting subcategory and suppose $\sM' \subset \sM$ is a partial silting subcategory. We define the \emph{rank} of $\sM'$, $\rk{\sM'}$, to be the rank of the split Grothendieck group $\Ksplit(\sM')$. Given a silting pair $(\sM,\sM')$, mutations $\rmu{\sM'}{(\sM)}$ and $\lmu{\sM'}{(\sM)}$ are called \emph{irreducible right mutation} and \emph{irreducible left mutation}, respectively, if $\rk{\sM'} = \rk{\sM} - 1$.
\end{definition}

We point out that if $\sM = \add(M)$ for some silting object $M$, then
\begin{align*}
 \rk{\sM} \,  &= \# \{\text{non-isomorphic indecomposable summands of } M\} , \\
 \rk{\sM'}    &= \# \{\text{non-isomorphic indecomposable summands of an additive generator of } \sM'\} .
\end{align*}

The following easy observation is certainly well known.

\begin{lemma} \label{lem:intersection-mutation}
If $(\sM,\sM')$ is a silting pair in $\sD$, then $\sM' = \sM \cap \rmu{\sM'}(\sM)$.
\end{lemma}

\begin{proof}
The inclusion $\sM'\subseteq\sM\cap\rmu{\sM'}(\sM)$ follows at once from the definition of the right mutation as $\rmu{\sM'}(\sM)=(\Sigma\inv \sM * \sM')\cap (\Sigma\inv \sM')\orth$: firstly $\sM' \subseteq \Sigma\inv \sM * \sM'$ is a triviality, and secondly $\sM' \subseteq (\Sigma\inv \sM')\orth$ follows from $\sM'$ partial silting.

For the other inclusion, let $A \in \Sigma\inv \sM * \sM'$. This object is, by definition, an extension of the form $\Sigma\inv M \xrightarrow{e}A \to M' \to M$ with $M'\in\sM'$ and $M\in\sM$. If, furthermore, $A\in\sM$, then $e=0$ as $\sM$ is silting. Hence $M'\cong A\oplus M$ which implies $A\in\sM'$, as $\sM'$ is an additive subcategory of $\sD$ and hence idempotent closed.
\end{proof}

Leaning on \cite{AI} but using a different convention --- see Remark~\ref{rem:conventions} --- we define an order on silting subcategories by $\sM \leq \sN$ if and only if $\Hom_{\sD}^{>0}(\sM,\sN)=0$. This is in fact a partial order; see \cite[Theorem 2.11]{AI}. We denote the poset of silting subcategories of $\sD$ by $\posetsilt(\sD)$.

\begin{example}
For any silting pair $(\sM,\sM')$, the mutations sit in the following chain of inequalities:  $\Sigma\inv\sM\leq\rmu{\sM'}(\sM)\leq\sM\leq\lmu{\sM'}(\sM)\leq\Sigma\sM$.
\end{example}

It will be convenient to have alternative descriptions of this partial order. For this, we recall two pieces of data equivalent to giving a silting subcategory.  By \cite{MSSS}, there is a bijection between silting subcategories $\sM$ of $\sD$ and bounded co-t-structures in $\sD$, which is given by
 $\sM \mapsto (\cosusp \Sigma^{-1} \sM, \susp \sM)$.

If $\sD=\Db(\Lambda)$ for some finite-dimensional algebra $\Lambda$, then by \cite{Koenig-Yang} one can also associate a bounded t-structure $(\sX_{\sM},\sY_{\sM})$ to $\sM$ as follows:
\begin{align*}
\sX_{\sM} & \coloneqq (\Sigma^{<0}\sM)\orth =
           \{X \in \sD \mid \Hom_{\sD}(M,\Sigma^i X)=0 \textrm{ for all } M\in \sM \textrm{ and } i > 0\} ; \\
\sY_{\sM} & \coloneqq (\Sigma^{\geq0}\sM)\orth =
           \{Y \in \sD \mid \Hom_{\sD}(M,\Sigma^i Y)=0 \textrm{ for all } M\in \sM \textrm{ and } i \leq 0\} .
\end{align*}

The partial order on the set of silting subcategories can be rephrased in terms of partial orders on the set of t-structures and co-t-structures; see \cite{Koenig-Yang} for example.

\begin{lemma} \label{lem:equivalent-po}
Let $\sM$ and $\sN$ be silting subcategories of $\sD$. Then the following conditions are equivalent, where (3) only applies if $\sD = \Db(\Lambda)$ for a finite-dimensional algebra $\Lambda$:
\begin{enumerate}[label=(\arabic*)~]
\item $\sM \leq \sN$, i.e.\ $\Hom^{>0}(\sM,\sN)=0$;
\item $\cosusp \sM \subseteq \cosusp \sN$ (equivalently $\susp \sM \supseteq \susp \sN$);
\item $\sY_{\sM} \subseteq \sY_{\sN}$ (equivalently $\sX_{\sM} \supseteq \sX_{\sN}$).
\end{enumerate}
\end{lemma}

The partial order is compatible with silting mutation:

\begin{lemma} \label{lem:mut-compat}
Let $\sM \leq \sN$ be silting subcategories of $\sD$ and suppose $\sK \subseteq \sM \cap \sN$. Then $\rmu{\sK}(\sM) \leq \rmu{\sK}(\sN)$ and $\lmu{\sK}(\sM) \leq \lmu{\sK}(\sN)$.
\end{lemma}

\begin{proof}
We prove the statement for right mutations, the statement for left mutations is dual. We need to show that $\Hom^{>0}(\rmu{\sK}(\sM),\rmu{\sK}(\sN))=0$, which using the invariant definitions amounts to
 $\Hom^{>0}( (\Sigma\inv \sM * \sK) \cap (\Sigma\inv \sK)\orth, (\Sigma\inv \sN * \sK) \cap (\Sigma\inv \sK)\orth ) = 0$. Indeed, the stronger
 $\Hom^{>0}( \Sigma\inv \sM * \sK, (\Sigma\inv \sN * \sK) \cap (\Sigma\inv \sK)\orth) = 0$ follows from these four facts:

\begin{enumerate}[label=(\arabic*)]
\item $\Hom^{>0}(\Sigma\inv \sM * \sK,\sK)=0$, since $\sM$ is silting and $\sK\subseteq\sM$.
\item $\Hom^{>0}(\Sigma\inv \sM,\Sigma\inv \sN)=\Hom^{>0}(\sM,\sN)=0$ as $\sM\leq\sN$.
\item $\Hom^1(\sK, (\Sigma\inv \sK)\orth)=0$ by the definition of the orthogonal subcategory.
\item $\Hom^{>1}(\sK, \Sigma\inv \sN)=\Hom^{>0}(\sK,\sN)=0$, since $\sN$ is silting and $\sK\subseteq\sN$. \qedhere
\end{enumerate}
\end{proof}

We have the following technical observation.

\begin{lemma} \label{lem:total_order}
Let $\sD$ be a Hom-finite, Krull-Schmidt triangulated category and $\sM$ a silting subcategory. Any finite subset of the poset $\posetsilt(\sD)$ is contained in an interval $[\Sigma^a\sM,\Sigma^b\sM]$.
\end{lemma}

\begin{proof}
We write $\rho \coloneqq \rk K_0(\sD)$. Let $\{\sN_1,\ldots,\sN_t\}\subseteq\posetsilt(\sD)$ be a finite subset. Choose additive generators for each silting subcategory in the set: $\sN_i = \add(\bigoplus_{j=1}^{\rho} N_{i,j})$. As $(\cosusp \Sigma^{-1} \sM, \susp \sM)$ is a bounded co-t-structure in $\sD$, for any $i,j$ there exists an integer $b_{i,j}$ with $N_{i,j} \in \cosusp \Sigma^{b_{i,j}} \sM$. Set $b\coloneqq\max b_{i,j}$; then $\cosusp \sN_i \subseteq \cosusp \Sigma^b \sM$ for each $i$, which by Lemma~\ref{lem:equivalent-po} means $\sN_i \leq \Sigma^b \sM$ for each $i$. Hence, $\Sigma^b\sM$ gives an upper bound. The index $a$ is found by a similar argument using $\susp \sM$.
\end{proof}

\begin{remark} \label{rem:conventions}
In the literature there are various conventions in use on naming of right and left mutations (and later, tilts) and choice of partial order.
Our partial orders are opposite to those in \cite{AI,Koenig-Yang,Woolf} but the same as those in \cite{King-Qiu,Qiu}.
What we call right mutation/right tilt corresponds to right mutation/right tilt in \cite{AI,Koenig-Yang}, backward tilt in \cite{King-Qiu,Qiu} and left tilt in \cite{Woolf}.
For us, right mutation/right tilt makes objects smaller in the partial order, whereas for \cite{AI,Koenig-Yang} it makes them larger.
We have chosen our partial order because of the convenient property $\Sigma\inv \sM \leq  \sM \leq \Sigma \sM$ for any silting subcategory $\sM$.
\end{remark}

Using the partial order on silting subcategories, we define an order on silting pairs:

\begin{definition}
Let $(\sM, \sM')$ and $(\sN,\sN')$ be two silting pairs of $\sD$, then
\[
(\sN,\sN') \leq (\sM,\sM') \iffdef \rmu{\sM'}(\sM) \leq \rmu{\sN'}(\sN) \leq \sN \leq \sM.
\]
\end{definition}

\begin{example}
  For any silting pair $(\sM,\sM')$, we have the following tautological chain:
$(\sM,\sM) \leq (\sM,\sM') \leq (\sM,0)$.
\end{example}

This is indeed a partial order on pairs:

\begin{lemma} \label{lem:partial_order_pairs}
The relation $\leq$ determines a partial order on silting pairs of $\sD$.
\end{lemma}
\begin{proof}
Reflexivity and transitivity are clear. For anti-symmetry suppose that $(\sN,\sN') \leq (\sM,\sM') $ and $(\sM,\sM') \leq (\sN,\sN')$. It follows from the definition that $\rmu{\sM'}(\sM) \leq \rmu{\sN'}(\sN) \leq \sN \leq \sM$ and $\rmu{\sN'}(\sN) \leq \rmu{\sM'}(\sM) \leq \sM \leq \sN$. Using the anti-symmetry of the partial order on silting objects we see that $\sM = \sN$ and $\rmu{\sM'}(\sM) = \rmu{\sN'}(\sN)$. Therefore, applying Lemma~\ref{lem:intersection-mutation},  $\sM' = \sM \cap \rmu{\sM'}(\sM) = \sN \cap \rmu{\sN'}(\sN) = \sN'$ and so $(\sN,\sN') = (\sM,\sM') $.
\end{proof}

In the following, the poset of silting pairs will play an important role. However, as soon as $\sD$ has more than one silting object, this poset does not possess a bottom element. For technical reasons, we need to artificially adjoin a bottom element, hence we make the following definition:

\begin{definition}
Let $(\posetpair(\sD), \leq)$ be the poset obtained from the poset of silting pairs by formally adjoining a bottom element $\hat{0}$. By abuse of terminology, from now on we shall call this the \emph{poset of silting pairs} and write $\posetpair(\sD)$.
\end{definition}


\section{Silting mutation versus admissible tilts}
\label{sec:compatibility}

\noindent
In this section, we describe the compatibility between silting mutation and certain tilts of t-structures at torsion pairs. This will facilitate the reader in translating statements in the language of silting pairs used in this article into the language of tilting torsion pairs, which is used in \cite{QW,Woolf}. As a consequence of the compatibility, we obtain two technical results which we shall need in Sections~\ref{sec:CW-properties} and \ref{sec:stability}.
Throughout this section we assume $\sD \coloneqq \Db(\Lambda)$, where $\Lambda$ is a finite-dimensional algebra.

\subsection{Admissible tilts} \label{sub:tilts}

We define the \emph{rank} of a subset of objects $\sS$ to be the rank of the K-group of the full triangulated subcategory generated by $\sS$ in $\sD$, cf. Definition~\ref{def:irreducible-mutation}.

Let $(\sX,\sY)$ be a bounded t-structure in $\sD$ with heart $\sH = \sX \cap \Sigma \sY$. If $(\cT,\cF)$ is a torsion pair in the heart $\sH$, then one obtains a new bounded  t-structure $(\sX',\sY')$ by
\begin{equation*} \label{def:admissible-tilt}
\sX' \coloneqq \clext{\Sigma^{n-1} \cT, \Sigma^n \cF \mid n \geq 0}
\quad \textrm{and} \quad
\sY' \coloneqq \clext{\Sigma^{n-1} \cT, \Sigma^n \cF \mid n < 0}.
\end{equation*}
The t-structure $(\sX',\sY')$ is called the \emph{right tilt of $(\sX,\sY)$ at the torsion pair $(\cT,\cF)$}; see \cite{HRS}.

Further suppose that $\sH$ is the heart of an algebraic t-structure, whose simple objects are $\simples(\sH) \coloneqq \{ S_1, \ldots, S_t \}$. If $\cT = \clext{S_{i_1},\ldots, S_{i_{\rho}}}_{\sH}$, where $\clext{\blank}_{\sH}$ denotes extension closure in $\sH$, then we call $(\sX',\sY')$ an \emph{admissible right tilt of rank $\rho$ of $(\sX,\sY)$}. Note that, when $\rho = 1$, such a tilt is called an \emph{irreducible right tilt of $(\sX,\sY)$}; see \cite{Koenig-Yang}.

\subsection{Compatibility} \label{sub:compatibility}

Let $\sM = \add(M)$ be a silting subcategory, with $M$ a basic silting object (i.e.\ all indecomposable summands are non-isomorphic); write $M = M_1 \oplus \cdots \oplus M_t$. By Theorem~\ref{thm:koenig-yang}, $\sM$ corresponds to an algebraic t-structure $(\sX_{\sM},\sY_{\sM})$ with heart $\sH_{\sM}$, which has $t$ non-isomorphic simple objects. We denote the set of (isomorphism classes) of simple objects of $\sH_{\sM}$ by $\simples(\sH_{\sM})$. By \cite{Koenig-Yang} the indecomposable summands of $M$ and the simple objects of $\sH_{\sM}$, appropriately enumerated, satisfy the following duality principle:
\[
\Hom(M_i,S_j) = \delta_{ij} \kk.
\]
We can now use this duality to define a torsion pair $(\cT_{\sM'},\cF_{\sM'})$ in $\sH_{\sM}$ as follows. Let
\[
\cT_{\sM'} \coloneqq   \clext{\simples(\sH_\sM) \cap (\sM')\orth}_{\sH_\sM}
\text{ and }
\cF_{\sM'} \coloneqq (\cT_{\sH_{\sM}})^{\perp_{\sH_{\sM}}}.
\]
The following lemma establishes the compatibility between silting mutation and admissible tilting. It is a straightforward generalisation of \cite[Theorem 7.12]{Koenig-Yang}.

\begin{lemma}
\label{lem:compatibility}
Let $\Lambda$ be a finite-dimensional algebra, $\sD\coloneqq\Db(\Lambda)$ its derived category and $\sK\coloneqq\sK^b(\proj{\Lambda})$ its bounded homotopy category. The bijection
\begin{align*}
   \{ \text{silting subcategories of $\sK$} \}
   & \stackrel{1-1}{\longleftrightarrow}
   \{ \text{algebraic t-structures in $\sD$} \} \\
   \sM
   & \longleftrightarrow
   (\sX_{\sM}, \sY_{\sM})
\end{align*}
of Theorem~\ref{thm:koenig-yang} is compatible with silting mutation and admissible tilting in the sense that the following diagram commutes:
\[
\xymatrix{
  \sM \ar@{<->}[rrr]^-{\textxy{K\"onig--Yang}} \ar@{~>}[d]_-{\textxy{mutate at } \sM'}    & & &
    (\sX_{\sM}, \sY_{\sM}) \ar@{~>}[d]^-{\textxy{tilt at } (\cT_{\sM'},\cF_{\sM'})} \\
\rmu{\sM'}(\sM) \ar@{<->}[rrr]_-{\textxy{K\"onig--Yang}}                      & & &
    (\sX_{\rmu{\sM'}(\sM)}, \sY_{\rmu{\sM'}(\sM)}).
}
\]
\end{lemma}

\subsection{An upper bound} \label{sub:bound}

The following technical lemma, or rather its corollary, will be used in
Section~\ref{sec:CW-properties}.

\begin{lemma}
\label{lem:intersection}
Suppose $(\sX,\sY)$ is an algebraic t-structure in $\sD$, with heart $\sH$.
For $i=1,2$, let $\sR_i\subseteq\simples(\sH)$ be a subset of simple objects, and denote by $(\sX_i,\sY_i)$ the admissible right tilt of $(\sX,\sY)$ induced by $\sR_i$.
Suppose there is an algebraic t-structure $(\sU,\sV)$ such that
\[
(\sU,\sV) \leq (\sX_1,\sY_1) \leq (\sX,\sY) \qquad\text{and}\qquad (\sU,\sV) \leq (\sX_2,\sY_2) \leq (\sX,\sY) .
\]
Write $\widetilde{\sR} \coloneqq \sR_1 \cup \sR_2$ and $(\widetilde{\sX},\widetilde{\sY})$ for the corresponding admissible right tilt of $(\sX,\sY)$. Then
\[  
(\sU,\sV) \leq (\widetilde{\sX},\widetilde{\sY}) \leq (\sX,\sY) .  
\]
\end{lemma}

\begin{proof}
Recall that $(\sU,\sV) \leq (\sX_i,\sY_i) \leq (\sX,\sY)$ if and only if $\sX \subseteq \sX_i \subseteq \sU$ for $i=1,2$. Thus, we need to show that $\widetilde{\sX} \subseteq \sU$. Let $\cT_i = \clext{\sR_i}_{\sH}$ and $\cF_i = \cT_i^{\perp_{\sH}}$ for $i=1,2$. Write $\widetilde{\cT}\coloneqq \clext{\widetilde{\sR}}_{\sH}$ and $\widetilde{\cF} \coloneqq \widetilde{\cT}^{\perp_{\sH}}$. Since
 $\widetilde{\sX} = \clext{ \Sigma^{n-1}\widetilde{\cT},\Sigma^n \widetilde{\cF} \mid n \geq 0 }$,
it is sufficient to verify $\Sigma^{n-1} \widetilde{\cT} \subseteq \sU$ and $\Sigma^n \widetilde{\cF} \subseteq \sU$ for $n\geq 0$ since $\sU$ is closed under extensions.

For the inclusion $\Sigma^n \widetilde{\cF} \subseteq \sU$, by definition and assumption we have
\[
\sX_i \coloneqq \clext{ \Sigma^{n-1} \cT_i, \Sigma^n \cF_i \mid n \geq 0 } \subseteq \sU .
\]
Now
 $\widetilde{\cF} = \widetilde{\cT}^{\perp_{\sH}}
                  =  \clext{\sR_1,\sR_2}_{\sH}^{\perp_{\sH}}
                  = \clext{\sR_1}_{\sH}^{\perp_{\sH}} \cap \clext{\sR_2}_{\sH}^{\perp_{\sH}}
                  = \cF_1 \cap \cF_2$.
Thus it follows that $\Sigma^n \widetilde{\cF} \subseteq \sU$ for each $n \geq 0$.

For the inclusion $\Sigma^{n-1} \widetilde{\cT} \subseteq \sU$: we have $\widetilde{\cT} = \clext{\sR_1,\sR_2}_{\sH} = \clext{\cT_1, \cT_2}_{\sH}$. Now, since $\sU$ is closed under extensions, it follows that $\Sigma^{n-1} \widetilde{\cT} \subseteq \sU$ for $n\geq 0$.
\end{proof}

\begin{corollary} \label{cor:upperbound}
Let $\Lambda$ be a finite-dimensional algebra.
Let $I$ be a set of silting pairs in $\sK^b(\proj{\Lambda})$ of the form $(\sK,\sK')$ for some fixed silting object $\sK$, and let
\[ \widetilde{\sK} \coloneqq \bigcap_{(\sK,\sK') \in I} \sK' . \]
For any silting object $\sM$ such that $\sM \leq \rmu{\sK'}(\sK)$ for all pairs $(\sK,\sK') \in I$, then $\sM \leq \rmu{\widetilde{\sK}}(\sK)$.
\end{corollary}

\begin{proof}
Note that the proof of Lemma~\ref{lem:intersection} holds for any finite number of admissible right tilts of the t-structure $(\sX,\sY)$. Hence the corollary follows by Lemma~\ref{lem:intersection} and the compatibility result, Lemma~\ref{lem:compatibility}.
\end{proof}

We also highlight the following link between silting pairs and (admissible) tilts.

\begin{lemma} \label{lem:silt-vs-torsion}
Let $(\sN,\sN') \leq (\sM,\sM')$ be silting pairs, i.e.\ $\rmu{\sM'}(\sM) \leq \rmu{\sN'}(\sN) \leq \sN \leq \sM$.
Let $(\cT_{\sM'},\cF_{\sM'})$ be the torsion pair for the admissible tilt corresponding to the mutation $\rmu{\sM'}(\sM)$.
Then there exists a torsion pair $(\cT,\cF)$ giving rise to a (possibly inadmissible) tilt from $(\sX_{\sM},\sY_{\sM})$ to $(\sX_{\sN},\sY_{\sN})$ such that $\cT \subseteq \cT_{\sM'}$.
\end{lemma}

\begin{proof}
We know that $\rmu{\sM'}(\sN) \leq \sN \leq \sM$. Using Lemma~\ref{lem:compatibility}, we see that the tilt between the corresponding t-structures $(\sX_{\sM}, \sY_{\sM})$ and $(\sX_{\rmu{\sM'}(\sM)}, \sY_{\rmu{\sM'}(\sM)})$ is induced by the torsion pair $(\cT_{\sM'}, \cF_{\sM'})$ in $\sH_{\sM}$ with $\cT_{\sM'} = \clext{\simples(\sH_\sM)\cap(\sM')\orth}_{\sH_{\sM}}$.
By \cite[Proposition~2.1]{Woolf}, there is a torsion pair $(\cT, \cF)$ in $\sH_{\sM}$ which induces the tilt from $(\sX_\sM, \sY_\sM)$ to $(\sX_\sN, \sY_\sN)$. Finally, using the aisle version of the partial order (see Lemma~\ref{lem:equivalent-po}), we get
\[ 
\cT = \sX_{\sN} \cap \Sigma \sY_{\sM} \subseteq \sX_{\rmu{\sM'}(\sM)} \cap \Sigma \sY_{\sM} = \cT_{\sM'}. \qedhere 
\]
\end{proof}


\section{Bongartz completion for silting subcategories}
\label{sec:bongartz}

\noindent
In this section we recall an analogue of the classical Bongartz completion of tilting modules for silting subcategories and connect it to the partial order on silting pairs. The next lemma collects several connections between intermediate silting objects and ordered extension closure. Let $\sD$ be a Krull-Schmidt triangulated category.

\begin{lemma} \label{lem:intermediate}
Let $\sM$, $\sN$ and $\sK$ be silting subcategories of $\sD$. Then
\begin{enumerate}[label=(\arabic*)]
\item $\sN \subseteq \Sigma\inv \sM * \sM$ if and only if $\Sigma\inv \sM \leq \sN \leq \sM$.
\item If $\Sigma\inv \sM \leq \sN \leq \sM$ and $\sK \subseteq \sN * \sM$, then $\sN \leq \sK \leq \sM$.
\item If $\sN\leq\sK\leq\sM$ then $\sM\cap\sN\subseteq\sK$.
\end{enumerate}
\end{lemma}

\begin{proof}
(1) This is \cite[Theorem 2.3]{IJY}.

(2) Since $\Sigma\inv \sM \leq \sN \leq \sM$, we have $\sN \leq \sM \leq \Sigma \sN$, and so, by (1), we have $\sM \subseteq \sN * \Sigma \sN$. Thus,
 $\sK \subseteq \sN * \sM \subseteq \sN * \sN * \Sigma \sN = \sN * \Sigma \sN$,
where the final equality follows from the associativity of the $*$ operation and the fact that silting subcategories are extension closed. Thus $\sN \leq \sK$, again by (1). The same argument using $\sN \subseteq \Sigma\inv \sM * \sM$ gives $\sK \leq \sM$.

(3) is \cite[Proposition 2.19]{AI}, taking care of the reversed conventions; see Remark~\ref{rem:conventions}.
\end{proof}

\begin{definition}
Let $\sM$ be a silting subcategory of $\sD$, and let $\sN'\subset\sM*\Sigma\sM$ be a partial silting subcategory. The \emph{right and left Bongartz completions} of $\sN'$ with respect to $\sM$ are
\begin{align*}
 \sB_r(\sN';\sM) &\coloneqq (\Sigma\inv \sM * \sN') \cap (\Sigma\inv\sN')\orth  && \text{right Bongartz completion.} \\
  \sB_l(\sN';\sM) &\coloneqq (\sN'* \sM) \cap {}\orth\Sigma\sN' && \text{left Bongartz completion,}
\end{align*}
\end{definition}

The above is a categorical version of the description using approximations from \cite[Proposition 6.1]{Wei}; see also \cite[Proposition 5.7]{DF}. This reference also contains the next lemma, which uses a triangulated version of Bongartz' completion argument \cite{Bongartz}.

\begin{lemma} \label{lem:bongartz}
Let $\sM$ be a silting subcategory of a Krull-Schmidt, Hom-finite triangulated category $\sD$, and $\sN' \subseteq \Sigma\inv \sM * \sM$ a partial silting subcategory. Consider the triangles
\begin{align*}
  & \Sigma\inv M \too B_r(M) \too N'_r(M) \rightlabel{f} M , \\
  & \Sigma\inv M \rightlabel{g} N'_l(M) \too B_l(M) \too M
\end{align*}
with $f$ a right $\sN'$-approximation of $M$ and $g$ a left $\sN'$-approximation of $\Sigma\inv M$. Then
\[
\add(\sN'\cup \{B_r(M) \mid M \in \sM\}) \text{ and } \add(\sN'\cup \{B_l(M) \mid M \in \sM\})
\]
are silting subcategories of $\sD$ lying in $\Sigma\inv \sM * \sM$.
\end{lemma}

The following lemma shows the equivalence of the two ways of completing $\sN'$ to a silting subcategory.  It also says that the choice of approximations in Lemma~\ref{lem:bongartz} do not affect the completions.

\begin{lemma}
Under the assumptions of Lemma~\ref{lem:bongartz}, we have
\[
\sB_r(\sN';\sM) = \add(\sN'\cup \{B_r(M)\mid M\in\sM\}) \text{ and } \sB_l(\sN';\sM) = \add(\sN'\cup\{B_l(M)\mid M\in\sM\}).
\]
\end{lemma}

\begin{proof}
Let $M\in \sM$. We have $B_r(M)\in \Sigma\inv \sM * \sN'$, and $N'\in \Sigma\inv \sM * \sN'$ trivially for each $N'\in \sN'$. Next, $B_r(M)\in(\Sigma^{-1} \sN')\orth$ follows immediately from applying $\Hom(\sN',\blank)$ to the triangle defining $B_r(M)$, and the defining property of right approximations. Finally, $N'\in(\Sigma^{-1} \sN')\orth$ is obvious from $\sN'$ partial silting. Altogether: $\add(\sN'\cup\{B_r(M)\mid M \in \sM\}) \subseteq \sB_r(\sN';\sM)$.

 For the other inclusion, suppose $B \in (\Sigma\inv \sM * \sN')\cap (\Sigma^{-1} \sN')\orth$. Then there is a triangle
$\Sigma\inv M_B \to B \to N'_B \xrightarrow{\alpha} M_B$ with $M_B \in \sM$ and $N'_B \in \sN'$. Applying $\Hom(\sN',-)$ to this triangle and using the fact that $B \in (\Sigma^{-1} \sN')\orth$ implies that $\alpha$ is a right $\sN'$-approximation of $M_B$, giving the other inclusion.

The equivalence for left Bongartz completions is analogous.
\end{proof}

\begin{remark}
The Bongartz completions of a silting category $\sM$ with respect to a subcategory $\sN' \subseteq \Sigma\inv \sM * \sM$ satisfy
 $\Sigma\inv \sM \leq \sB_r(\sN';\sM) \leq \sB_l(\sN';\sM) \leq \sM$.

In other words, the Bongartz completions are in the interval $[\Sigma\inv \sM,\sM]$. Moreover, the left completion is the unique silting subcategory containing $\sN'$ which is maximal inside $[\Sigma\inv \sM,\sM]$. Likewise, the right completion is the unique silting subcategory containing $\sN'$ which is minimal inside $[\Sigma\inv \sM,\sM]$.
\end{remark}

The following is immediate from the categorical definitions of the Bongartz completions.

\begin{lemma} \label{lem:trivial-completion}
Suppose $\sM$ is a silting subcategory of $\sD$ and $\sM' \subseteq \sM$ is a subcategory. Then $\sB_r(\sM';\sM) = \rmu{\sM'}(\sM)$ and $\sB_l(\sM';\sM)=\sM$.
\end{lemma}

In the following proposition, we provide equivalent formulations of the definition of the partial order on silting pairs which use Bongartz completion. Before stating the proposition, we isolate a technical lemma.

\begin{lemma}\label{lem:technical}
Let $\sD$ be a Hom-finite, Krull-Schmidt triangulated category with a silting object. Suppose $(\sN,\sN') \leq (\sM,\sM')$ are silting pairs, $\sB'\subseteq \sN'$ and $\sB \coloneqq \sB_l(\sB';\sM)$. Then $\rmu{\sB'}(\sB) \leq \rmu{\sN'}(\sN)$.
\end{lemma}

\begin{proof}
We observe the following chain of inclusions:
\[ \begin{array}{rc l @{\hspace{0.1\textwidth}} r}
\rmu{\sB'}(\sB)
    & \subseteq &
\Sigma\inv \sB * \sB'                                                 & \quad (\textrm{by the definition of $\rmu{\sB'}(\sB)$})  \\
    & \subseteq &
\Sigma\inv \sB * \sN'                                                 & \quad (\textrm{by $\sB' \subseteq \sN'$}) \\
    & \subseteq &
\Sigma^{-2} \sM * \Sigma\inv \sM * \sN'                    & \quad (\textrm{by } \Sigma\inv \sM \leq \sB \leq \sM) \\
    & \subseteq &
\Sigma^{-2} \sM * \Sigma\inv \sM * \rmu{\sN'}(\sN)   & \quad (\textrm{by } \sN' \subseteq \rmu{\sN'}(\sN)) \\
    & \subseteq &
\multicolumn{2}{l}{
(\Sigma^{-2}\rmu{\sN'}(\sN) * \Sigma\inv \rmu{\sN'}(\sN))*(\Sigma\inv \rmu{\sN'}(\sN) * \rmu{\sN'}(\sN))*\rmu{\sN'}(\sN)
     \hfill (*)             } \\
    & \subseteq &
\Sigma^{-2} \rmu{\sN'}(\sN) * \Sigma\inv \rmu{\sN'}(\sN) * \rmu{\sN'}(\sN) & \quad (\textrm{by } \rmu{\sN'}(\sN) * \rmu{\sN'}(\sN) = \rmu{\sN'}(\sN)) \\
    & \subseteq &
\cosusp{\rmu{\sN'}(\sN)},
\end{array} \]
where $(*)$ follows from $\Sigma\inv \rmu{\sN'}(\sN) \leq \Sigma \inv \sM \leq \rmu{\sN'}(\sN)$. Now $\rmu{\sB'}(\sB)\subseteq\cosusp{\rmu{\sN'}(\sN)}$ implies $\cosusp\rmu{\sB'}(\sB)\subseteq\cosusp{\rmu{\sN'}(\sN)}$,
and so $\rmu{\sB'}(\sB) \leq \rmu{\sN'}(\sN)$ by Lemma~\ref{lem:equivalent-po}.
\end{proof}

\begin{proposition} \label{prop:po-equivalent}
Let $\sD$ be a Hom-finite, Krull-Schmidt triangulated category with a silting object.
Suppose $(\sM,\sM')$ and $(\sN,\sN')$ are silting pairs of $\sD$. Then the following statements are equivalent:
\begin{enumerate}[label=(\arabic*)~~]
\item $(\sN,\sN') \leq (\sM,\sM')$, i.e.\ $\rmu{\sM'}(\sM) \leq \rmu{\sN'}(\sN) \leq \sN \leq \sM$;
\item $\sM' \subseteq \sN' \subseteq \Sigma\inv \sM * \sM$ and $\sB_r(\sN';\sM) = \rmu{\sN'}(\sN)$;
\item $\sM' \subseteq \sN' \subseteq \Sigma\inv \sM * \sM$ and $\sB_l(\sN';\sM) = \sN$.
\end{enumerate}
\end{proposition}

\begin{proof}
We make silent use of the property $\sM' = \sM \cap \rmu{\sM'}(\sM)$ from Lemma~\ref{lem:intersection-mutation}.

(1) $\implies$ (3).
From $(\sN,\sN') \leq (\sM, \sM')$ and $\Sigma\inv \sM \leq \rmu{\sM'}(\sM)$ we get $\Sigma\inv \sM \leq \sN \leq \sM$, hence
$\sN'\subseteq \sN \subseteq \Sigma\inv \sM * \sM$ by Lemma~\ref{lem:intermediate}(1).
Applying Lemma~\ref{lem:intermediate}(3) to $\rmu{\sM'}(\sM) \leq \sN \leq \sM$ gets us $\sM' = \sM \cap \rmu{\sM'}(\sM) \subseteq \sN$, and applying it to $\rmu{\sM'}(\sM) \leq \rmu{\sN'}(\sN) \leq \sM$ gives $\sM' = \sM \cap \rmu{\sM'}(\sM) \subseteq \rmu{\sN'}(\sN)$. Hence $\sM' \subseteq \sN \cap \rmu{\sN'}(\sN) = \sN'$.

By the definition of $\sB \coloneqq \sB_l(\sN';\sM)$ we have $\sB \subseteq \sN * \sM$; hence $\sN \leq \sB \leq \sM$ from Lemma~\ref{lem:intermediate}(2).
Moreover, $\rmu{\sN'}(\sN) \leq \rmu{\sN'}(\sB)$ by Lemma~\ref{lem:mut-compat}. Applying Lemma~\ref{lem:technical} with $\sB' = \sN'$ gives $\rmu{\sN'}(\sB) \leq \rmu{\sN'}(\sN)$. Hence, $\rmu{\sN'}(\sB) = \rmu{\sN'}(\sN)$.
Using Hom-finiteness and the fact that $\sD$ has a silting object, $\sN'$ is functorially finite in $\rmu{\sN'}(\sB) = \rmu{\sN'}(\sN)$, and so we can perform the corresponding left mutations, i.e.\
\[
\sB = \lmu{\sN'}(\rmu{\sN'}(\sB)) = \lmu{\sN'}(\rmu{\sN'}(\sN) ) = \sN,
\]
giving condition $(3)$, as claimed.

(3) $\implies$ (1). Again, let $\sB \coloneqq \sB_l(\sN';\sM)$. It suffices to prove $(\sB,\sN') \leq (\sM,\sM')$, i.e.\
\[
\rmu{\sM'}(\sM) \leq \rmu{\sN'}(\sB) \leq \sB \leq \sM.
\]
The two rightmost inequalities we get directly since $\sB$ is a Bongartz completion in $\Sigma\inv \sM * \sM$ and $\rmu{\sN'}(\sB) \leq \sB$ always. We are left to show $\rmu{\sM'}(\sM) \leq \rmu{\sN'}(\sB)$, which by definition of the partial order and of right mutations amounts to the vanishing of
\[  \Hom^{>0}(\rmu{\sM'}(\sM),\rmu{\sN'}(\sB))
  = \Hom^{>0}( (\Sigma\inv \sM * \sM') \cap (\Sigma\inv \sM')\orth, (\Sigma\inv \sB * \sN') \cap (\Sigma\inv \sN')\orth) .\]
We have $\Hom^{>0}(\sM',\sN')=0$ from $\sM'\subseteq\sN'\subseteq\sB$ and $\sB$ silting. Moreover, $\sN' \subseteq \Sigma\inv \sM * \sM$ and $\sM$ silting gives $\Hom^{>0}(\Sigma\inv \sM, \sN')\subseteq\Hom^{>0}(\Sigma\inv \sM, \Sigma\inv \sM * \sM)=0$. Hence
\begin{align*}
\Hom^{>0}(\rmu{\sM'}(\sM),\rmu{\sN'}(\sB))
   &= \Hom^{>0}( (\Sigma\inv \sM * \sM') \cap (\Sigma\inv \sM')\orth, \Sigma\inv \sB \cap (\Sigma\inv \sN')\orth) \\
   &= \Hom^{>0}( (\Sigma\inv \sM * \sM') \cap (\Sigma\inv \sM')\orth, \Sigma\inv (\sB \cap (\sN')\orth)) \\
   &= \Hom^{\geq0}( (\Sigma\inv \sM * \sM') \cap (\Sigma\inv \sM')\orth, (\sN' * \sM) \cap {}\orth(\Sigma \sN') \cap (\sN')\orth) \\
   &\stackrel{*}{=} \Hom^{\geq0}( (\Sigma\inv \sM * \sM') \cap (\Sigma\inv \sM')\orth, \sM \cap {}\orth(\Sigma \sN') \cap (\sN')\orth) \\
   &= \Hom^0( (\Sigma\inv \sM * \sM') \cap (\Sigma\inv \sM')\orth, \sM \cap {}\orth(\Sigma \sN') \cap (\sN')\orth) \\
   &= \Hom^0( \sM' \cap (\Sigma\inv \sM')\orth, \sM \cap {}\orth(\Sigma \sN') \cap (\sN')\orth) \\
   &\subseteq \Hom^0( \sM', \sM \cap (\sN')\orth) \\
   & \subseteq \Hom^0( \sN', (\sN')\orth)  = 0
\end{align*}
where the marked equality follows from $\sN' \cap (\sN')\orth = 0$, which implies that objects in $\sN' * \sM$ are direct summands of objects of $\sM$, hence in $\sM$ (partial silting subcategories are idempotent closed). The next two equalities follow from $\sM$ silting, and the final inclusion from $\sM'\subseteq\sN'$.

$(1) \Longleftrightarrow (2).$ Since $\sD$ is a Hom-finite, Krull-Schmidt triangulated category with a silting object, writing $\widetilde{\sM}= \rmu{\sM'}(\sM)$ and $\widetilde{\sN}=\rmu{\sN'}(\sN)$, we can rephrase the inequalities $\rmu{\sM'}(\sM) \leq \rmu{\sN'}(\sN) \leq \sN \leq \sM$ as
$\widetilde{\sM} \leq \widetilde{\sN} \leq \lmu{\sN'}(\widetilde{\sN}) \leq \lmu{\sM'}(\widetilde{\sM})$.
Now, the dual arguments to above using right Bongartz completion instead of left Bongartz completion give the desired conclusions.
\end{proof}


\section{The poset of silting pairs is a CW poset}
\label{sec:CW-poset}

\noindent
It is well known that any poset gives rise to a simplicial complex, the order complex of the poset: the vertices of the complex are the elements of the poset, and the faces are the finite chains, i.e.\ finite, totally ordered subsets of the poset. However, the order complex is more finely subdivided than is necessary and this makes calculations longer than they need to be. Furthermore, we would like to have a cellular structure which mirrors the structure of the silting objects and their mutations more closely. For example, the silting quiver of \cite{AI} should be the 1-skeleton and ``higher" mutations should correspond to higher dimensional faces.

Therefore, for our applications of the poset of silting pairs to the stability manifold, we want the structure of a regular CW complex instead (this is a CW complex such that all attaching maps are homeomorphisms). In particular, we want to discuss homotopy properties of topological spaces (arising from silting pairs, and from stability conditions) and regular CW complexes behave very well with regard to homotopy theory.

In \cite[Definition 2.1]{BjoernerEJC}, Bj\"orner introduced the class of CW posets, which correspond to regular CW complexes \cite[Proposition 3.1]{BjoernerEJC}. In this section we prove the following main theorem which gives conditions under which we obtain such a poset.

\begin{theorem}
\label{thm:CW-poset}
Let $\sD$ be a Hom-finite, Krull-Schmidt triangulated category with a silting object. Further suppose that for each silting subcategory $\sM$ the interval $[\Sigma\inv \sM,\sM]$ in $\posetsilt(\sD)$ is finite. Then the silting pair poset $\posetpair(\sD)$ is a CW poset.
\end{theorem}

We recall from \cite{Aihara} that a Hom-finite, Krull-Schmidt triangulated category $\sD$ with a silting object is called \emph{silting-discrete} if for each silting subcategory $\sM$ and each natural number $k$, the interval $[\Sigma^{-k} \sM,\sM]$ in $\posetsilt(\sD)$ is finite. We therefore get the following corollary.

\begin{corollary}
If $\sD$ is a silting-discrete triangulated category then the silting pair poset $\posetpair(\sD)$ is a CW poset.
\end{corollary}

In light of Theorem~\ref{thm:CW-poset}, we proceed to define the CW complex of silting pairs for a triangulated category.

\begin{definition}
Let $\sD$ be a triangulated category. In the case that $\posetpair(\sD)$ is a CW poset, we denote the induced CW complex by $\silt{\sD}$ and call it the \emph{silting pairs CW complex}.
\end{definition}

In the next subsection we define CW posets and describe how the corresponding regular CW complex is constructed.  We state criteria for a given poset to be a CW poset and verify these criteria, under the assumptions of Theorem~\ref{thm:CW-poset} thus proving the theorem.
In the final subsection we turn to our main class of examples and check that discrete derived categories satisfy the hypotheses of Theorem~\ref{thm:CW-poset}.

\subsection{CW posets}\label{sec:CW-posets}

We start with a brief digression on basic notions relating to posets. For the following definitions, see \cite{BjoernerWachs}. Given a poset $P$ and two elements $x,y\in P$, then $y$ \emph{covers} $x$ if $y$ is an immediate successor of $x$, i.e.\ $x<y$ and if $x<z\leq y$ then $z=y$. The poset $P$ is called
\begin{itemize}[leftmargin=1.5em]
\item \emph{bounded} if $P$ has top $\hat{1}$ and bottom $\hat{0}$ elements, i.e.\
      $\hat{0} \leq x \leq \hat{1}$ for all $x\in P$;
\item \emph{semimodular} if it is finite, bounded and whenever two distinct elements $u,v\in P$
      both cover $x\in P$ there is a $z\in P$ which covers both $u$ and $v$;
\item \emph{totally semimodular} if it is finite, bounded and all intervals $[x,y]$ in $P$ are semimodular.
\end{itemize}

Given two elements $x<y$, the \emph{length} of the interval $[x,y]$ is the maximal number $l$ such that there is a chain $x=z_0<z_1<\ldots<z_{l-1}<z_l=y$.

\begin{definition}
A poset $P$ is called a \emph{CW poset} if $P$ has a bottom element $\hat0$, contains at least two elements and the order complexes of open intervals $(\hat0,x)$ are homeomorphic to spheres for all $x\in P$, $x\neq\hat0$.
\end{definition}

Given a CW poset $P$, we outline the construction of the corresponding regular CW complex following the proof of \cite[Proposition 3.1]{BjoernerEJC}. First observe that, by the CW poset property, every $x \in P$ has a well defined rank $r(x)$ given by the length of the interval $[\hat0,x]$. For any $i \geq 1$ we denote the set of elements of rank $i$ by $P_i = \{x \in P \mid r(x) = i \}$.

We proceed inductively, starting with $\cK_0 \coloneqq P_1$ as a disjoint union of points.  Suppose we have constructed a regular CW complex $\cK_{s-1}$ with face poset $F(\cK_{s-1}) \cong \bigcup_{i \leq s} P_i$. For every $x \in P_{s+1}$, the open interval $(\smash{\hat0},x)$ corresponds to a regular CW subcomplex $\cK_x \subset \cK_{s-1}$ and by definition of CW poset, the order complex of $F(\cK_x) \setminus \{\hat{0}\}$ is homeomorphic to an $(s-1)$-sphere. We attach an $s$-cell to $\cK_{s-1}$ for each $x \in P_{s+1}$ using these homeomorphisms. The resulting CW complex $\cK_s$ is regular, since the attaching maps are homeomorphisms onto their images, and satisfies $F(\cK_{s}) \cong \bigcup_{i \leq s+1} P_i$. This construction produces a regular CW complex with face poset isomorphic to $P$.

\begin{remark}
This means that in our examples, we get an order preserving bijection between closed cells of $\silt{\sD}$ ordered by containment and silting pairs in $\posetpair(\sD)$. In particular, the higher dimensional faces of $\silt{\sD}$ correspond to ``higher" mutations, as mentioned in the introduction to this section.
\end{remark}

In this article, we apply the following criteria to check whether a poset is a CW poset. It is an immediate corollary of work of Bj\"orner and Wachs; see the proof for details.

\begin{proposition} \label{prop:CW}
Let $P$ be a poset satisfying the following conditions:
\begin{enumerate}
\item $P$ has a bottom element $\hat 0$, and contains at least one other element,
\item every interval $[x,y]$ of length two has cardinality four,
\item for every $x \in P$ the interval $[\hat 0,x]$ is totally semimodular.
\end{enumerate}
Then $P$ is a CW poset.
\end{proposition}

\begin{proof}
In \cite[Proposition~2.2]{BjoernerEJC}, Bj\"orner proves a more general statement than the one given here, with `totally semimodular' in clause (3) replaced by `shellable'. We refrain from defining shellability here; see the introduction of \cite{BjoernerWachs} for a purely combinatorial definition.
We mention that shellability is a notion originating from topology. For instance, the order complex of a shellable poset has the homotopy type of a wedge of $r$-spheres if all maximal chains have length $r$.
For our purposes, we can bypass this notion as every totally semimodular poset is shellable by \cite[Proposition 2.3 and Corollary 5.2]{BjoernerWachs}.
\end{proof}

In the definition of $\posetpair(\sD)$, we formally adjoined a bottom element $\hat{0}$ to the set of silting pairs. Thus, we only have to check the other two conditions.

\subsection{Cardinality of length two intervals}\label{sec:length-two}

We prove property (2) of Proposition~\ref{prop:CW}. The results of this subsection hold for Hom-finite, Krull-Schmidt triangulated categories $\sD$ that have a silting object.
Recall the definition of rank, $\rk{\sM'}$, of a partial silting subcategory $\sM'$ from Definition~\ref{def:irreducible-mutation}.

\begin{proposition}
\label{prop:length-two}
Suppose that $(\sM, \sM')$ and $(\sN,\sN')$ are silting pairs in $\sD$ such that $(\sN,\sN') \leq (\sM,\sM')$ and $\rk{\sN'} = \rk{\sM'} + 2$. Then there are precisely four silting pairs $(\sK,\sK')$ such that $(\sN,\sN') \leq (\sK, \sK') \leq (\sM,\sM')$.
\end{proposition}

\begin{proof}
If a silting pair $(\sK,\sK')$ in $\sD$ satisfies $(\sN,\sN') \leq (\sK, \sK') \leq (\sM,\sM')$ then, by Proposition~\ref{prop:po-equivalent}, we have $\sM' \subseteq \sK' \subseteq \sN'$.  Since $\rk{\sN'} = \rk{\sM'} +2$, $\sN'$ is additively generated by $\sM'$ and only two further indecomposable objects, up to isomorphism; write  $N_1$ and $N_2$ for these. Thus there are precisely four possibilities for $\sK'$ satisfying the required inclusion: $\sM'$, $\add(\sM',N_1)$, $\add(\sM',N_2)$ and $\sN'$.

For each possible $\sK'$, we have $\sM' \subseteq \sK'$ and $\sK' \subseteq \Sigma\inv \sM * \sM$, this latter inclusion coming from the fact that $\sK'\subseteq \sN'$ and $(\sN,\sN') \leq (\sM,\sM')$. By Proposition~\ref{prop:po-equivalent}, $\sK\coloneqq \sB_l(\sK';\sM)$ is the unique completion of $\sK'$ satisfying $(\sK,\sK') \leq (\sM, \sM')$.

We need to check $(\sN,\sN') \leq (\sK,\sK')$, i.e.\ that $\rmu{\sK'}(\sK) \leq \rmu{\sN'}(\sN) \leq \sN \leq \sK$. The Bongartz completion $\sK$ is formed by taking the additive closure of $K'$ together with objects $B_l$ of triangles
\[
\Sigma\inv M \to K'_l \to B_l \to M
\]
where $K'_l \in \sK'$. Thus $\sK \subseteq \sK' * \sM \subseteq \sN' * \sM \subseteq \sN * \sM$, whence by Lemma~\ref{lem:intermediate}(2), $\sN \leq \sK \leq \sM$. Now, since $\sK' \subseteq \sN'$, we can apply Lemma~\ref{lem:technical} to get $\rmu{\sK'}(\sK) \leq \rmu{\sN'}(\sN)$.
\end{proof}

\begin{lemma} \label{lem:intlength}
For any interval $I \coloneqq [(\sN, \sN'), (\sM,\sM')] \subseteq \posetpair$, the length $\ell(I) = \rk{\sN'}-\rk{\sM'}$.
\end{lemma}
\begin{proof}
Let $I$ be an interval of length $\ell$ and let
\[
(\sN, \sN') = (\sN_0,\sN'_0) < (\sN_1,\sN'_1) < \cdots < (\sN_{\ell},\sN'_{\ell})  = (\sM,\sM') 
\]
be a strictly increasing chain of this maximal length. Using Proposition~\ref{prop:po-equivalent} we observe that the partial silting subcategories form a nested sequence $\sN'_{\ell} \subseteq \sN'_{\ell-1} \subseteq \dots \subseteq \sN'_0$. If $\rk{\sN'_{j}} \geq \rk{\sN'_{j+1}} +2$ for some $j \in \{ 0, \dots , \ell - 1\} $ then as in Proposition~\ref{prop:length-two} we could construct intermediate silting pairs contradicting the maximality of the length of the chain.  If $\rk{\sN'_{j}} = \rk{\sN'_{j+1}}$ for some $j \in \{ 0, \dots , \ell - 1\} $, then the inclusion forces $\sN'_{j} = \sN'_{j+1}$. It then follows that
\[
\sN_j = \sB_l(\sN'_j; \sN_{j+1}) = \sB_l(\sN'_{j+1}; \sN_{j+1}) = \sN_{j+1}
\]
where the first equality comes from Proposition~\ref{prop:po-equivalent} and the third one from Lemma~\ref{lem:trivial-completion}. However this would contradict the strictly increasing property of the chain. Therefore the only possibility is that $\rk{\sN'_{j}} = \rk{\sN'_{j+1}}+1$ for all $j \in \{ 0, \dots , \ell - 1\} $. It follows that $\rk{\sN'}-\rk{\sM'} = \rk{\sN'_0}-\rk{\sN'_\ell} = \ell$.
\end{proof}

Combining Proposition~\ref{prop:length-two} and Lemma~\ref{lem:intlength} immediately gets us the following corollary.

\begin{corollary}
\label{cor:length-two}
Every interval in the poset $\posetpair(\sD)$ of length two has cardinality four.
\end{corollary}

\subsection{Total semimodularity } \label{sec:finite-and-totally-semimodular}

In this subsection we assume that $\sD$ satisfies all of the hypotheses from Theorem~\ref{thm:CW-poset}.
\begin{lemma} \label{cor:finite}
For every $(\sK,\sK') \in \posetpair$, the interval $[\hat 0,(\sK,\sK')]$ is finite.
\end{lemma}

\begin{proof}
By the hypotheses we know that there are a finite number of silting subcategories $\sN$ such that $\Sigma\inv \sK \leq \sN \leq \sK$. Since each of these silting subcategories has a finite number of different partial silting subcategories, there are at most a finite number of silting pairs $(\sN, \sN')$ such that $\Sigma\inv \sK \leq \rmu{\sN'}(\sN) \leq \sN \leq \sK$. In other words the interval $[\hat 0,(\sK,0)]$ is finite. It follows that $[\hat 0,(\sK,\sK')]$ is finite, since this is a subinterval.
\end{proof}

\begin{proposition} \label{prop:totally-semimodular}
Every interval $[x,y] \subseteq \posetpair$ is totally semimodular.
In particular, the interval $[\hat 0,(\sK,\sK')]$ is totally semimodular for every $(\sK,\sK') \in \posetpair$.
\end{proposition}

\begin{proof}
Let  $a=(\sA,\sA')$, $x=(\sX,\sX')$ and $y=(\sY,\sY')$ be silting pairs with $a \in [x,y] \subseteq \posetpair$. By definition, we have
\[
\rmu{\sY'}(\sY) \leq \rmu{\sA'}(\sA) \leq \rmu{\sX'}(\sX) \leq \sX \leq \sA \leq \sY.
\]
Suppose that $u=(\sU,\sU')$ and $v=(\sV,\sV')$ are silting pairs which cover $a$ and lie in the interval $[x,y]$.
Since $a < u$ and $a<v$ in $\posetpair$, by definition we have the inequalities
\[
\rmu{\sU'}(\sU)  \leq \rmu{\sA'}(\sA) \leq \sA \leq \sU
\quad \textrm{and} \quad
\rmu{\sV'}(\sV)  \leq \rmu{\sA'}(\sA) \leq \sA \leq \sV.
\]

Set $\sB'\coloneqq \sU' \cap \sV'$, $\sB\coloneqq \sB_l(\sB';\sY)$ and write $b=(\sB,\sB')$. We claim that $u \leq b$; the corresponding claim $v \leq b$ is analogous. Since $\sB' \subseteq \sU'$, the argument of Proposition~\ref{prop:length-two} shows that $\sU \leq \sB \leq \sY$.
Now, since $\sB' \subseteq \sU'$, apply Lemma~\ref{lem:technical} to see that $\rmu{\sB'}(\sB) \leq \rmu{\sU'}(\sU)$, giving the desired chain of inequalities. Thus $u \leq b$.

We have $\sY'\subseteq \sU'$ and $\sY' \subseteq \sV'$ by Proposition~\ref{prop:po-equivalent}, thus $\sY' \subseteq \sB'$. Analogously, we obtain $\sB' \subseteq \Sigma\inv \sY *  \sY$. Thus, by Proposition~\ref{prop:po-equivalent}, we have $(\sB,\sB') \leq (\sY,\sY')$, i.e $b \leq y$. Now $x\leq u$ and $u \leq b$, so by transitivity of $\leq$, we get $x \leq b \leq y$, i.e.\ $b \in [x,y]$.

Now only the minimality aspect of the covering property remains.
Suppose $u \leq w \leq b$ for some $w=(\sW,\sW')$, i.e.\ $(\sU,\sU') < (\sW,\sW') \leq (\sB,\sB')$. By Proposition~\ref{prop:po-equivalent}, $\sB' \subseteq \sW' \subseteq \sU'$.
Now suppose $\rk \sA' = \rho$, then $\rk \sU' = \rk \sV' = \rho - 1$. Since $(\sU,\sU') \neq (\sV,\sV')$, it follows that $\sU' \neq \sV'$. Hence, $\rk \sB' = \rho - 2$.
Hence $\ind{\sU'}\setminus\ind{\sB'}$ consists of only one object, up to isomorphism. In particular, $\sW'=\sB'$ or $\sW'=\sU'$.

Suppose $\sW'=\sB'$ and consider $(\sW,\sW') \leq (\sB,\sB')$. Then by Proposition~\ref{prop:po-equivalent}, we have $\sW = \sB_l(\sW';\sB) = \sB_l(\sB';\sB) = \sB$, where the last equality follows from Lemma~\ref{lem:trivial-completion}.
Applying the same argument to $(\sU,\sU') \leq (\sW,\sW')$ in the case that $\sW' = \sU'$ giving $\sW = \sU$. Hence the minimality of the covering property holds for $u$; analogously for $v$.
\end{proof}

Putting Lemma~\ref{cor:finite}, Proposition~\ref{prop:totally-semimodular} and Corollary~\ref{cor:length-two} together yields Theorem~\ref{thm:CW-poset}.

\subsection{Finiteness}
We now prove the finiteness condition from Theorem~\ref{thm:CW-poset} holds for the bounded derived categories of derived discrete algebras of finite global dimension. In fact, we actually prove something stronger, which we shall use again later. The following proposition is a converse to Lemma~\ref{lem:total_order} in the case that $\sD = \Db(\LLambda)$ is the derived category of a derived-discrete algebra.

\begin{proposition} \label{prop:finiteness}
Let $\sM$ be a silting subcategory of $\Db(\LLambda)$ and $k \in \IN$. The set of silting subcategories $\sN$ such that $\Sigma^{-k} \sM \leq \sN \leq \sM$ is finite.
\end{proposition}

\begin{remark}
By \cite[Proposition 3.8]{Aihara}, Proposition~\ref{prop:finiteness} means that $\Db(\LLambda)$ is silting-discrete.
\end{remark}

\begin{proof}[Proof of Proposition~\ref{prop:finiteness}]
To avoid minus signs we actually prove the equivalent statement that the set of silting subcategories $\sN$ such that $\sM \leq \sN \leq \Sigma^k \sM$ is finite.
By Lemma~\ref{lem:Z-in-silting}, there are indecomposable objects $M_0 \in \sM$ and $N_0 \in \sN$ which lie in the $\cZ$ component of the AR quiver of $\Db(\Lambda)$. The inequalities $\sM \leq \sN \leq \Sigma^k \sM$ together imply that $\Hom(M_0,\Sigma^i N_0) =0= \Hom(N_0,\Sigma^{k+i} M_0)$ for all $i>0$.
Using Proposition~\ref{prop:Z-hammocks} to examine the Hom-hammocks for $\Hom(\Sigma^{-i} M_0,\blank)$ and $\Hom(\blank,\Sigma^{k+i}M_0)$ for $i>0$ excludes indecomposable objects in the shaded regions in Figure~\ref{fig:finiteness} from being possible candidates for $N_0$.

\begin{figure} \label{fig:finiteness}
 \parbox{\textwidth}{
    \parbox{0.3\textwidth}{
      \centerline{$\cZ^0$}
      \includegraphics[width=0.3\textwidth]{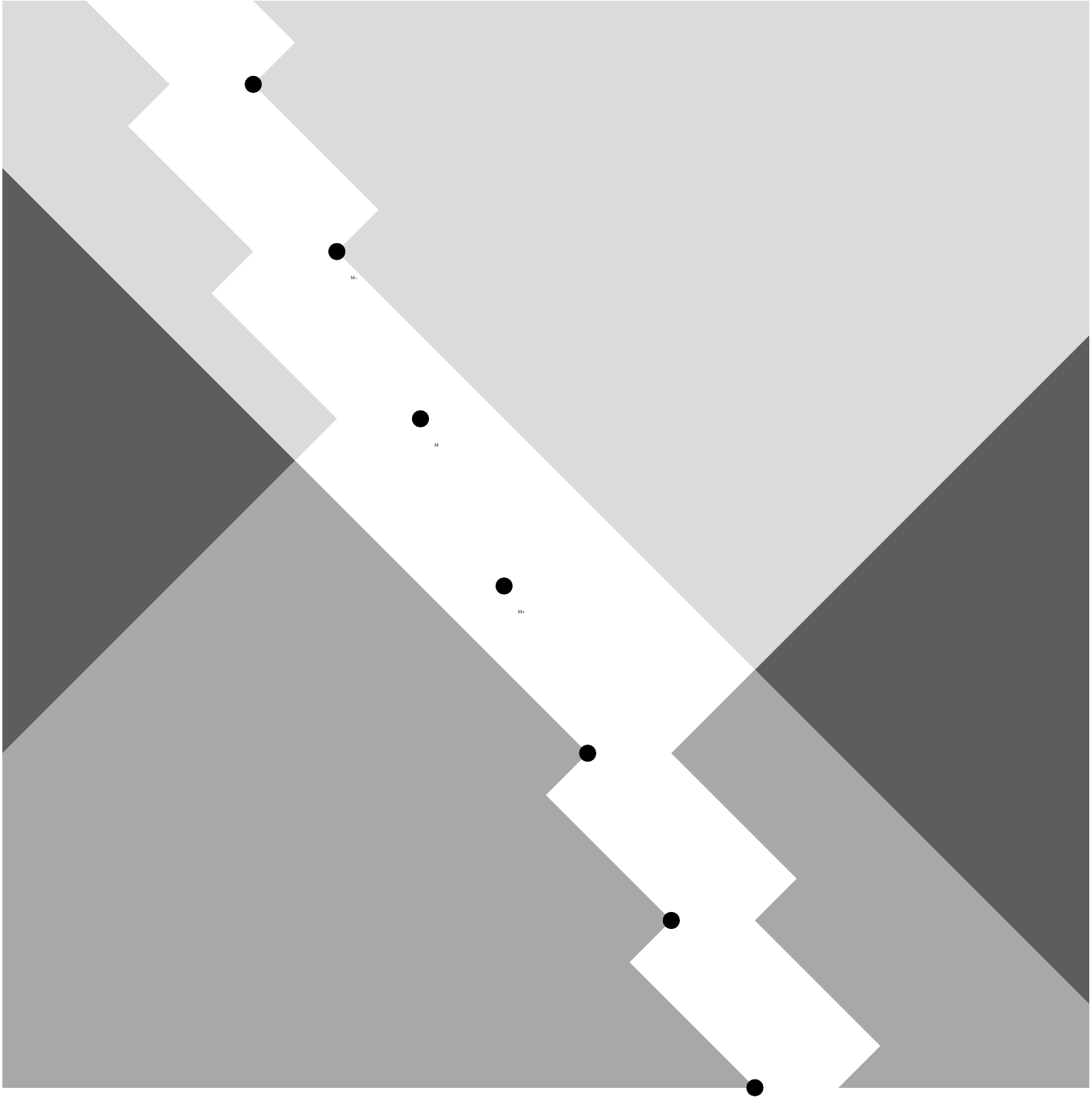}
    }
    \hfill
    \parbox{0.3\textwidth}{
      \centerline{$\cZ^1$}
      \includegraphics[width=0.3\textwidth]{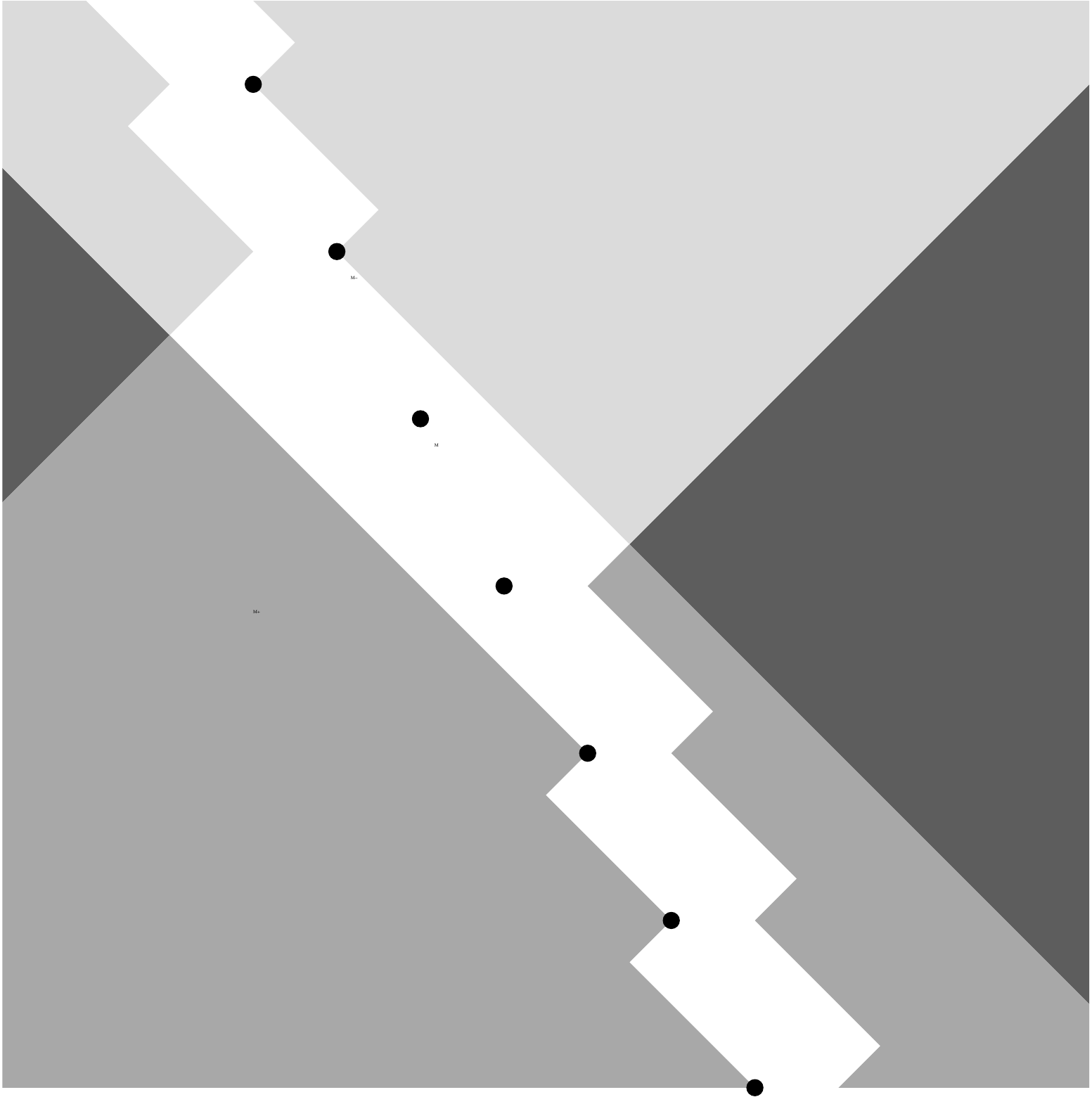}
    }
    \hfill
    \parbox{0.3\textwidth}{
      \centerline{$\cZ^2\cdots\cZ^{r-1}$}
      \includegraphics[width=0.3\textwidth]{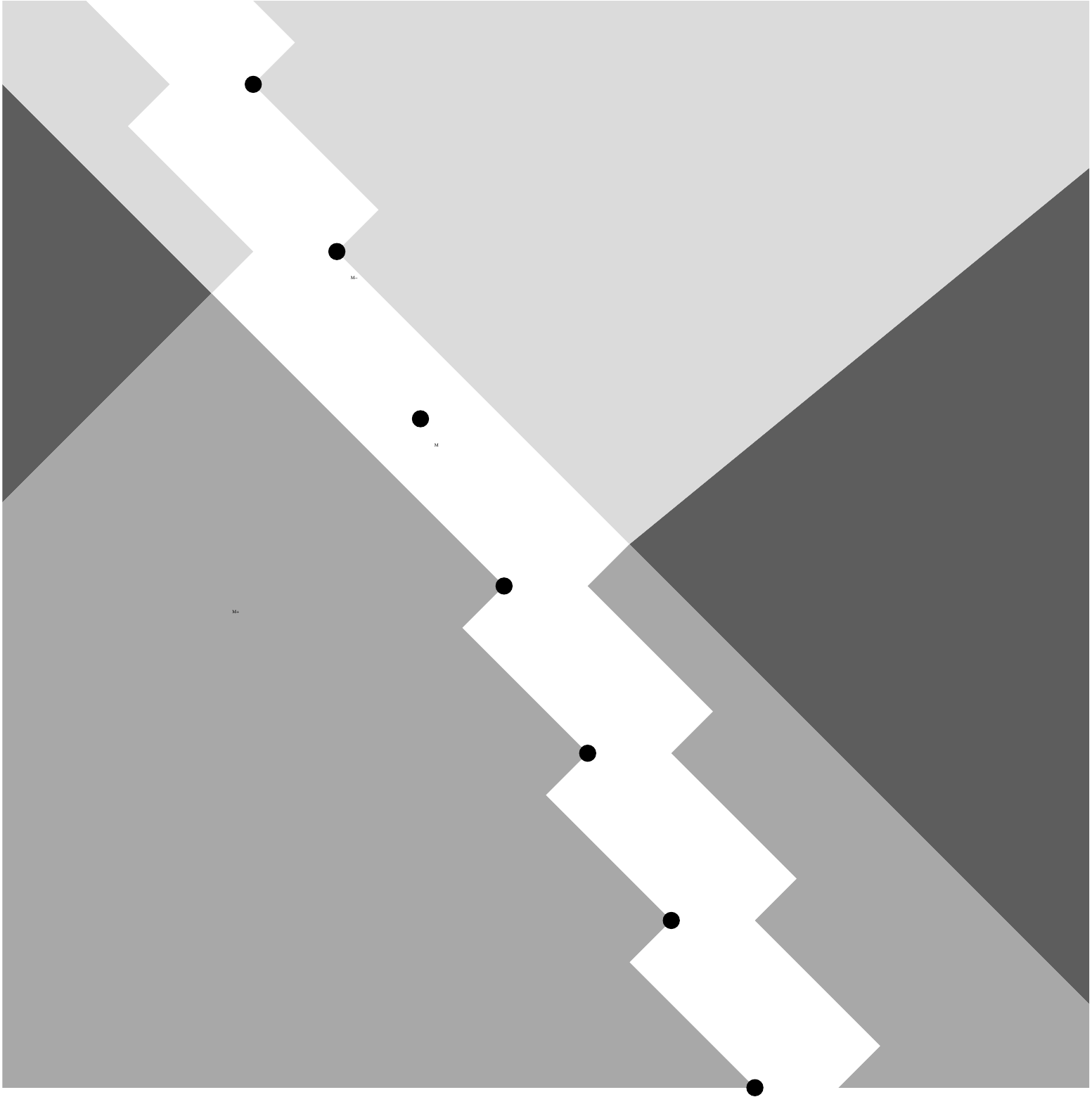}
    }
 }
  \caption{\label{fig:prop-finiteness}
           For the proof of Proposition~\ref{prop:finiteness}. \newline
           Hammocks $\Hom(\blank,\Sigma^{k+i}M)\neq0$ in light grey;
           $\Hom(\Sigma^{-i}M,\blank)\neq0$ in grey, and the intersection in dark grey.
           Note that the right-hand picture applies to all of $\cZ^2,\dots,\cZ^{r-1}$;
           the objects given are for $\cZ^{r-1}$. \newline
           The black dots indicate the objects $\ldots,\Sigma^{-2}M,\Sigma\inv M,M,\Sigma M,\Sigma^2M,\ldots$
           Their additive closure is $\thick{}{M} = \add(\Sigma^i M \mid i\in\IZ)$.
          }
\end{figure}

The remaining unshaded regions of the AR quiver of the $\cZ$ component contain finitely many indecomposable objects up to suspensions. Consider one such object, $Z\in\cZ$ and assume $\Sigma^jZ \in \sN$ for some $j\in\IZ$.
Since the t-structure $(\sX_{\sM},\sY_{\sM})$ is bounded, there is a largest $p\in\IZ$ such that $\Sigma^p Z \notin \sX_\sM$.
Thus, for any silting subcategory $\sN\geq\sM$, i.e.\ $\sX_\sN\subseteq\sX_\sM$, we have $\Sigma^p Z \notin \sX_\sN$. Since $\sN\subset\sX_\sN$ due to $\sN$ silting, $\Sigma^j Z\in\sN$ implies $p<j$. In other words,
 $p$ is a lower bound on the possible suspensions of objects $\Sigma^j Z\in\ind{\cZ}\cap\sN$ with $\sN\geq\sM$.

For the upper bound, consider $\Sigma \tau Z$, where $\tau$ is the Auslander--Reiten translation in $\Db(\LLambda)$. By the boundedness of $(\sX_{\Sigma^k \sM},\sY_{\Sigma^k \sM})$, there is a largest $q' \in \IZ$ such that $\Sigma^{q'} (\Sigma\tau Z) \notin \sX_{\Sigma^k \sM} =(\Sigma^{<0}\Sigma^k\sM)\orth$. Taking $q < \min\{0, q'\}$, by Nakayama (equivalently, Serre) duality there exists $t>0$ such that
\[
0 \neq \Hom(\Sigma^k \sM, \Sigma^{t +q +1}\tau Z) \cong \Hom(\Sigma^{t+q} Z,\Sigma^k \sM)^*
     = \Hom(\Sigma^t Z, \Sigma^{k-q} \sM)^* ,
\]
where $(\blank)^*$ denotes the dual vector space.
This implies $\Sigma^t Z \notin \sN$ since $q < 0$ and $N \leq \Sigma^k M$.  It is easy to check that $\Sigma^{t+i} Z$ cannot be an object of $\sN$ for each $i \geq 0$.

Hence there are only finitely many possible indecomposable objects in the $\cZ$ component which may lie in some silting subcategory $\sN$ such that $\sM \leq \sN \leq \Sigma^k \sM$. Let $\cZ_{\sM}$ be the additive subcategory of $\Db(\Lambda)$ generated by these objects.

For each $Z \in \ind{\cZ_{\sM}}$, we can perform silting reduction at $\thick{}{Z}$, which is functorially finite by Proposition~\ref{prop:embedding-A}. Furthermore, the same proposition asserts
$\thick{}{Z}\orth \simeq \Db(\kk A_{n+m-1})$.
Since the aisles $\sX_{\sM}$ and $\sX_{\Sigma^k \sM}$ are  bounded, there are integers $a$ and $b$ so that for each indecomposable object $U \in \thick{}{Z}\orth$ we have  $\Sigma^i U \in \sX_{\sM}$ for $i>a$ and $\Sigma^i U \in \sX_{\Sigma^k \sM}$ for $i>b$. Similarly for the co-aisles $\sY_{\sM}$ and $\sY_{\Sigma^k \sM}$, with $\Sigma^{-i}$ in place of $\Sigma^i$. As the path algebra $\kk A_{n+m-1}$ is representation-finite, there are only finitely many possible complements to $Z$ which also satisfy the orthogonality conditions imposed by $\sM$ and $\Sigma^k \sM$ in $\thick{}{Z}\orth$. Theorem~\ref{thm:silt-red} gives a bijection between silting subcategories of $\Db(\Lambda)$ containing $Z$ and silting subcategories of $\thick{}{Z}\orth$, hence there are only finitely many silting subcategories $\sN$ of $\Db(\Lambda)$ such that $Z \in \sN$ that also satisfy $\sM \leq \sN \leq \Sigma^k \sM$. Since there are only finitely many possible $Z$, we obtain the claim.
\end{proof}

\begin{remark}
This argument can be modified to prove the analogous result for $\sD = \Db(\kk Q)$ where $Q$ is a simply-laced quiver of Dynkin type.
\end{remark}


\section{Contractibility of the silting pairs CW complex for $\Db(\Lambda)$}
\label{sec:CW-properties}

\noindent
In this section, we come to the main theorem of the article.

\begin{theorem} \label{thm:CW-contractible}
The silting pairs CW complex $\silt{\Db(\LLambda)}$ is contractible.
\end{theorem}

\begin{proof}
We write $X \coloneqq \silt{\Db(\LLambda)}$ for the silting pairs complex. As $X$ is a CW complex, by Whitehead's theorem, contractibility is equivalent to all homotopy groups $\pi_i(X)$ being trivial; see \cite[Theorem 4.5]{Hatcher}. Hence we have to prove that every map $f\colon S^i \to X$ from an $i$-sphere to $X$ is homotopic to a constant map. Furthermore, with $X$ and $S^i$ two CW complexes, every such map $f$ is homotopic to a cellular map; see for example \cite[Theorem~4.8]{Hatcher}. Hence we can assume that the image $f(S^i)$ is contained in $j$-cells for $j \leq i$. Each cell of $S^i$ is compact so this image meets finitely many cells in $X$ by \cite[Proposition~A.1]{Hatcher}. The cells correspond to silting pairs and so, forgetting for the moment the partial silting subcategory of the pair, we obtain a finite set of silting subcategories which we denote $\siltsub{f}$.

By Lemma~\ref{lem:total_order}, there exist an $a\in\IZ$ and a silting subcategory $\sM$ such that $\siltsub{f} \subset [\sM, \Sigma^a \sM]$.
By Proposition~\ref{prop:finiteness}, the interval $[\sM,\Sigma^a \sM]$ is finite.
We fix the choice of such an $\sM$ and consider the finite subposet of $\posetsilt(\sD)$
 \[ \rho(f) \coloneqq
    \{ \sK \text{ a silting subcategory} \mid \sM \leq \sK \leq \sN \text{ for any } \sN \in \siltsub{f} \}  . \]
We proceed inductively. If $\lvert \rho(f) \rvert =1$, then $\siltsub{f} = \{\sM\}$ and so the image $f(S^i)$ is a subset of the top-dimensional cell $(\sM, 0)$. This is contractible, and so $f$ is homotopic to a constant map.

If $\lvert \rho(f) \rvert >1$, then we choose some maximal element $\sK$ of the poset $\rho(f)$, and look at the set $I$ of all pairs $(\sK,\sK')$ for which the corresponding cells intersect $f(S^i)$.
We construct a new silting pair
\[ (\sK, \widetilde{\sK}) \coloneqq \big( \sK , \bigcap_{(\sK,\sK') \in I} \sK' \big) \]
and denote by $C$, the corresponding (closed) cell in the silting pairs CW complex. In fact there are several CW complexes that will be of interest, which we now list together to fix notation and aid readability:

\medskip\noindent
\begin{tabular}{@{} p{0.05\textwidth} @{} p{0.95\textwidth} @{}}
$X$:   & the silting pairs CW complex, \\
$C$: & the closed cell in $X$ corresponding to the pair $(\sK, \widetilde{\sK})$, \\
$N$:  & the union of all closed cells in $X$ which do not contain the point $(\sK, \sK)$, \\
$A$:  & the union of all closed cells in $C$ which do not contain the point $(\sK, \sK)$. \\
\end{tabular}

\medskip\noindent
We construct a deformation retraction of $C$ onto $A$: Since $X$ is a regular CW complex, the cell $C$ is homeomorphic to a ball $B^k$ and its boundary is homeomorphic to a sphere $S^k$. The complement of $A$ in the boundary of $C$ is open and contractible to the point $(\sK,\sK)$ so $A$ is homeomorphic to a (closed) hemisphere of $S^k$. There is a deformation retraction of $B^k$ onto this hemisphere, induced by the projection along a diameter. Therefore, we can write down a deformation retract of the cell $g\colon C \times [0,1] \to C$ onto $A$ by composing this deformation retract with the two homeomorphisms.

We wish to extend this to a homotopy on the whole of $X$. We do this in two steps: first, since the restriction of $g(-,t)$ to the intersection $C \cap N = A$ is the identity map for all $t \in [0,1]$, we can extend $g$ trivially to obtain a deformation retraction $\widetilde g\colon C \cup N \times [0,1] \to C \cup N$ onto $N$. Since $C \cup N$ is a union of closed cells in $X$, then $(X, C \cup N)$ is a CW pair and so has the homotopy extension property; see \cite[Proposition 0.16]{Hatcher}. In particular, $\widetilde g(-,0) = \id_{C \cup N}$, so we can extend $\widetilde g$ to a homotopy $h\colon X \times [0,1] \to X$ such that $h(-,0) = \id_{X}$.

The composition
 $F \coloneqq h \circ (f \times \id) \colon S^i \times [0,1] \to X$
gives a homotopy between $f = F(-,0)$ and $f_1 \coloneqq F(-,1)$.
We consider which cells the image $f_1(S^i)$ intersects. Since $h$ is an extension of $\widetilde g$ which is a deformation retraction onto $N$, we know that $h(-,t)$ restricted to $N$ is the identity for all $t$. Therefore a point $x$ in the interior of $N$ is in $f_1(S^i)$ if and only if it is in $f(S^i)$. In other words, the images $f_1(S^i)$ and $f(S^i)$ only differ on $\overline{X \setminus N}$.

Any point in $\overline{X \setminus N}$ is contained in a closed cell $(\sL,\sL')$ of $X$ which contains the vertex $(\sK,\sK)$. Therefore $(\sK,\sK) \leq (\sL,\sL')$ which, in particular, implies that $\sK \leq \sL$. Using the maximality of $\sK$, it follows that if $f(x) \in f(S^i)$ is in this cell, then $\sL = \sK$ and it follows that $f(x) \in C$. By construction, $h$ retracts $C$ onto $A$ and so $f_1(x) \in A $. Putting this together, we have that $f_1(x) \in A$ if $f(x) \in \overline{X \setminus N}$ and  $f_1(x) = f(x)$ otherwise.  Note that the silting subcategory of the cells in $A$  correspond to silting subcategories in the half open interval $[\rmu{\widetilde{\sK}}(\sK) , \sK )$, however by Corollary~\ref{cor:upperbound}, $\sM \leq \rmu{\widetilde{\sK}} \sK \leq \sK$.
Therefore the same $\sM$ is a lower bound of $\siltsub{f_1}$, and $ \rho(f_1) \subseteq \rho(f) \setminus \{\sK\} $ so $\lvert \rho(f) \rvert > \lvert \rho(f_1) \rvert$.

The result then follows by induction.
\end{proof}

\begin{corollary}
The silting quiver, in the sense of \cite{AI}, of $\Db(\LLambda)$ is connected.
\end{corollary}

Suppose $\sD$ is a triangulated category with a silting subcategory $\sM$. Any silting subcategory $\sN$ of $\sD$ with $\sN \subseteq \Sigma\inv \sM * \sM$ is called a \emph{two-term $\sM$-silting subcategory}.

\begin{corollary} \label{cor:contractible}
Let $\sM$ be a silting subcategory of $\Db(\LLambda)$. Then the CW complex of two-term $\sM$-silting subcategories is contractible.
\end{corollary}

\begin{proof}
The silting subcategory $\sM$ generates a cell of the silting pairs CW complex corresponding to the interval $\Sigma\inv \sM * \sM$. This interval is a closed ball in the silting pairs CW complex, and as such, is contractible.
\end{proof}

\begin{remark} \label{rem:cluster-connection}
Corollary~\ref{cor:contractible} may be of wider interest in cluster-tilting theory. There are well-known connections between two-term silting objects and cluster tilting theory; see \cite{Bruestle-Yang} for an overview of the theory. As such, Corollary~\ref{cor:contractible} may be considered a result regarding `cluster structures' for algebras that until now occur outside cluster theory.
\end{remark}


\section{The stability manifold of $\Db(\LLambda)$}
\label{sec:stability}

\noindent
We provide a recap of the basic notions of stability conditions following \cite{Bridgeland}. Here $\sD$ is an arbitrary triangulated category.
Let $\IH \coloneqq \{r\exp(i \pi \phi) \mid r > 0 \textrm{ and } 0 < \phi \leq 1 \} \subseteq \IC$ and write $\phi(z)\in (0,1]$ for the phase of $z\in\IC$. Note that $\pi\phi(z)$ is the argument of $z \in \IC$.

A \emph{stability function} on an abelian category $\sA$ is a group homomorphism $Z \colon K_0(\sA) \to \IC$ such that $Z(A)\in\IH$ for all $0 \neq A \in \sA$.
An object $0 \neq A \in \sA$ is \emph{semistable} with respect to $Z$ if every subobject $0 \neq A' \subset A$ satisfies $\phi(A') \leq \phi(A)$, where we write $\phi(A) \coloneqq \phi(Z(A))$.

A \emph{Harder-Narasimhan (HN) filtration} of $0 \neq A \in \sA$ is a finite chain of subobjects
\[
0 = A_0 \subset A_1 \subset \cdots \subset A_{n-1} \subset A_n = A,
\]
whose factors $F_i = A_i/A_{i-1}$ are semistable objects with
$
\phi(F_1) > \phi(F_2) > \cdots > \phi(F_n).
$
A stability function $Z$ is said to satisfy the \emph{HN property} if every non-zero object of $\sA$ admits a HN filtration.

Rather than define stability conditions on triangulated categories, we refer to the following proposition, which provides a description fitting our framework better.

\begin{proposition}[{\cite[Proposition 5.3]{Bridgeland}}]
To give a stability condition on a triangulated category $\sD$ is equivalent to giving a bounded t-structure in $\sD$ and a stability function on its heart which satisfies the HN property.
\end{proposition}

As in Section~\ref{sec:compatibility}, we write $\simples(\sH)$ for the set of simple objects of an abelian category $\sH$.

Now let $\sD = \Db(\Lambda)$ be the bounded derived category of a finite-dimensional algebra $\Lambda$ of finite global dimension. Then the K\"onig--Yang correspondences (Theorem~\ref{thm:Koenig-Yang}) associate to any silting subcategory $\sM$ of $\sD$ a t-structure $(\sX_\sM,\sY_\sM)$; we denote its heart by $\sH_\sM$. Moreover, $\sH_\sM$ is the module category of a finite-dimensional algebra, and thus satisfies the HN property.

Now we restrict to a discrete derived category $\sD = \Db(\LLambda)$.
Then $\sH_\sM$ has finitely many simple objects by Proposition~\ref{prop:length}. Therefore, the above characterisation allows us to define stability conditions on $\sD$ by mapping the finitely many simple objects $\simples(\sH_\sM)\to\IH$.

The set of stability conditions on a triangulated category $\sD$ is denoted by $\stab{\sD}$. In \cite{Bridgeland}, $\stab{\sD}$ was topologised using a (generalised) metric. (We mention that all stability conditions on $\Db(\LLambda)$ are locally finite and numerical.) The following is the most basic fact about the stability space.

\begin{theorem}[{\cite[Theorem 1.2]{Bridgeland}}]
The topological space $\stab{\sD}$ is a complex manifold.
\end{theorem}

For general triangulated categories, it is only known that each connected component of $\stab{\sD}$ is locally homeomorphic to a linear subspace of $K_0(\sD)$. However, for $\sD=\Db(\LLambda)$ we show in Theorem~\ref{thm:stab-contractible} that the stability manifold is contractible, so particularly connected. Moreover, each point has a neighbourhood corresponding to stability functions on some bounded heart. Hence, $\stab{\Db(\LLambda)}$ is connected and has dimension $\rk K_0(\sD)$.

\subsection{An embedding $\silt{\Db(\LLambda)} \hookrightarrow \stab{\Db(\LLambda)}$}

Throughout this section $\sD = \Db(\LLambda)$. Let $U(\sH) \subseteq \stab{\sD}$ be the subset of stability conditions which have heart $\sH$. Since for discrete derived categories all hearts are length categories by Proposition~\ref{prop:length}, it follows from \cite[Lemma 5.2]{Bridgeland} that
\[ U(\sH) \cong \IH^{n+m}  \]
where $n+m = \rk K_0(\sD)$ is the number of simple objects in the heart.

Define
$f\colon \silt{\sD} \too \stab{\sD}$
as follows:
Let $(\sM,\sM')$ be a silting pair --- this is an element of the poset $\posetpair(\sD)$ and hence a point of $\silt{\sD}$. By Theorem~\ref{thm:Koenig-Yang}, the silting subcategory $\sM$ gives rise to a t-structure and thus to a heart $\sH_M$. We define a stability function on $\sH$ by mapping simple objects to the upper half-plane:
\[ f(\sM,\sM') \coloneqq (\sH_\sM, \charge_{\sM,\sM'}) \quad\text{with}\quad
   \charge_{\sM,\sM'} \colon \simples(\sH_\sM) \to \IH, \quad
   S \mapsto
                                   \begin{cases}
                                      i = e^{i\pi/2} &  \text{if } S \notin (\sM')\orth  \\
                                     -1             &  \text{if } S \in    (\sM')\orth
                                    \end{cases}
\]

Thus, simple objects $S$ with $\Hom(\sM',S)\neq 0$ get mapped to the imaginary unit $i$, and the other simple objects get mapped to $-1$.

Since $\silt{\sD}$ is homeomorphic to the order complex of $\posetpair(\sD) \setminus \{\hat{0}\}$, an arbitary point is given by the data $(\sM_\fl,\sM'_\fl,a_\fl)$, where
\[ (\sM_\fl,\sM'_\fl) = \big( (\sM_0,\sM_0) \leq (\sM_1,\sM'_1) \leq \ldots \leq (\sM_{k-1},\sM'_{k-1})  \leq (\sM_k,0)  \big) \]
is a chain of silting pairs of maximal length, and $a_\fl=(a_0,\ldots,a_k)$ is a convex combination, i.e.\ a tuple of non-negative real numbers which sum to one.
Let $l\in\{0,\ldots,k\}$ be minimal with $a_l>0$. Then we set
\[ f(\sM_\fl,\sM'_\fl,a_\fl) \coloneqq \big(\sH_{\sM_l}, \, \sum_{j=0}^k a_j \charge_{\sM_j,\sM_j'} \big) . \]

In the rest of the section we shall employ the following shorthand notation.

\begin{notation}
For silting pairs $(\sM,\sM')$  and $(\sN,\sN')$ we write $\rmu{}\sM=\rmu{\sM'}(\sM)$ and $\rmu{}\sN=\rmu{\sN'}(\sN)$.
\end{notation}

\begin{lemma} \label{lem:cellclosure}
Let $(\sN,\sN') \leq (\sM,\sM')$ be silting pairs. Then $f(\sM,\sM') \in \overline{U(\sH_\sN)}$.
\end{lemma}

\begin{proof}
By Lemma~\ref{lem:silt-vs-torsion}, there are torsion pairs $(\cT,\cF)$ and $(\cT_{\sM'},\cF_{\sM'})$, respectively, corresponding to the tilts from $(\sX_{\sM},\sY_{\sM})$ to $(\sX_{\sN},\sY_{\sN})$ and to $(\sX_{\rmu{}{\sM}},\sY_{\rmu{}{\sM}})$. Moreover $\cT \subseteq \cT_{\sM'}$.
The stability function of $f(\sM,\sM')$ has phase $1$ on precisely the simple objects $S$ with $\Hom(\sM',S) = 0$, which lie in $\cT_{\sM'}$. Thus, by \cite[Corollary~2.8]{Woolf} and Proposition~\ref{prop:length} (which allows us to apply \cite[Corollary~2.8]{Woolf}), we find $f(\sM,\sM') \in U(\sH_\sM) \cap \overline{U(\sH_\sN)}$.
\end{proof}

The torsion class of the admissible tilt occurring in the above proof was defined in the proof of Lemma~\ref{lem:silt-vs-torsion}. It is $\cT_{\sM'} = \clext{\simples(\sH_\sM)\cap(\sM')\orth}_{\sH_{\sM}}$.

\begin{corollary} \label{cor:real}
Let $(\sN,\sN') \leq (\sM,\sM')$ be silting pairs and suppose $S \in \simples(\sH_{\sN}) \cap (\sN')\orth$. Then $\charge_{\sM,\sM'}(S) \in \IR$.
\end{corollary}

\begin{proof}
We have
$
\simples(\sH_{\sN}) \cap (\sN')\orth \subset \cT_{\sN'} = \sX_{\rmu{}\sN} \cap \Sigma \sY_\sN \subset \sX_{\rmu{}\sN} \cap \Sigma \sY_\sM \subset \sX_{\rmu{}\sM} \cap \Sigma \sY_\sM = \cT_{\sM'}.
$
\end{proof}

\begin{lemma}
$Z_{\sM_\fl,\sM'_\fl}$ is a well defined stability function.
\end{lemma}

\begin{proof}
The point $f(\sM_\fl,\sM'_\fl,a_\fl)$ is defined on the heart $\sH_{\sM_l}$, so let $S$ be a simple object in this heart. We need to show that the stability function maps $S$ into $\IH$.

First suppose $\Hom(\sM'_l,S)\neq 0$. Then $\charge_{\sM_l,\sM'_l}(S) = i$ and hence $f(\sM_j,\sM'_j) \in \overline{U(\sH_{\sM_l})}$ for all $j \geq l$ by Lemma~\ref{lem:cellclosure}. Therefore, $\charge_{\sM_j,\sM'_j}(S)$ is in the closure of the upper half-plane for all $j \geq l$; see \cite[Proposition~2.11]{Woolf}. It follows that the convex sum $\charge_{\sM_\fl,\sM'_\fl}(S)$ is a point in the open upper half-plane.

Now we consider the set of simple objects $\simples(\sH_{\sM_l}) \cap (\sM'_l)\orth$. From Lemma~\ref{lem:compatibility} there are torsion pairs giving rise to admissible tilts as follows:
\begin{align*}
(\cT \coloneqq \clext{\simples(\sH_{\sM_l}) \cap (\sM'_l)\orth}, \cF)
     & \text{ corresponding to } (\sX_{\sM_l},\sY_{\sM_l}) \leadsto (\sX_{\rmu{}{\sM_l}},\sY_{\rmu{}{\sM_l}}), \\
(\cT_j \coloneqq \clext{\simples(\sH_{\sM_j}) \cap (\sM'_j)\orth}, \cF_j)
     & \text{ corresponding to } (\sX_{\sM_j},\sY_{\sM_j}) \leadsto (\sX_{\rmu{}{\sM_j}},\sY_{\rmu{}{\sM_j}}).
\end{align*}
Applying Lemma~\ref{lem:silt-vs-torsion} gives $\cT \subseteq \cT_j$ for all $j > l$.

By definition, $\charge_{\sM_j,\sM'_j}$ has phase $1$ on each generator of $\cT_j$, and therefore, on each object in $\cT_j$. Using $\cT \subseteq \cT_j$, we have that $\charge_{\sM_j,\sM'_j}$ has phase $1$ for all objects of $\cT$ and in particular for all simple objects in $S \in \simples(\sH_{\sM_l}) \cap (\sM'_l)\orth$. It follows that $\charge_{\sM_{*},\sM'_{*}}(S) \in \IH$.
\end{proof}

\begin{remark}
We observe that the map $f$ is well defined: In the order complex of $\posetpair(\sD) \setminus \{\hat{0}\}$, two points $(\sM_\fl,\sM'_\fl,a_\fl)$ and $(\sN_\fl,\sN'_\fl,b_\fl)$ are the same if and only if $a_j =b_j$ for all $j\in\{0,\ldots,k\}$ and $(\sM_j,\sM'_j)=(\sN_j,\sN'_j)$ when $a_j =b_j \neq 0$. It is then clear from the definition that $f(\sM_\fl,\sM'_\fl,a_\fl) = f(\sN_\fl,\sN'_\fl,b_\fl)$.
\end{remark}

\begin{lemma}
$f$ is continuous.
\end{lemma}

\begin{proof}
Since $\silt{\sD}$ is a first-countable topological space, it is sufficient to prove sequential continuity. Take any sequence $(x_n)$ of points in $\silt{\sD}$ which converge to a point $x \in \silt{\sD}$. Using the finiteness property proved in Proposition~\ref{prop:finiteness} it is straightforward to see that $x$ is contained in at most a finite number simplices of the order complex of $\posetpair(\sD) \setminus \{\hat{0}\}$. We can therefore partition the sequence into a finite number of subsequences whose terms eventually lie in the closure of one of the simplices containing $x$. On the image of each such simplex under $f$, the heart is fixed and the stability function varies linearly and is therefore continuous. Hence every subsequence converges to $f(x)$.
\end{proof}

\begin{lemma}
$f$ is injective.
\end{lemma}

\begin{proof}
Given any point $(\sM_\fl,\sM'_\fl,a_\fl)$ in the silting pairs CW complex, we can assign an integer $\# a_\fl \in (0, k]$ to be the number of $a_i$ which are non-zero. For two points $(\sM_\fl,\sM'_\fl,a_\fl)$ and $(\sN_\fl,\sN'_\fl,b_\fl)$ define an integer $d \coloneqq \max\{ \#  a_\fl,  \#  b_\fl\}$. We use this integer as the index for a proof by induction. Assume $f(\sM_\fl,\sM'_\fl,a_\fl) = f(\sN_\fl,\sN'_\fl,b_\fl)$.

First suppose that $d=1$. It follows that there is exacly one non-zero $a_i$ and $b_j$ and therefore these are both equal to $1$. Since the images are the same we see that $\sH_{\sM_l} = \sH_{\sN_{l'}}$ and $\charge_{\sM_l,\sM_l'} = \charge_{\sN_{l'},\sN_{l'}'}$.
The hearts determine the bounded t-structures, which by the K\"onig--Yang correspondences, Theorem~\ref{thm:koenig-yang}, determine the silting subcategories, so that $\sM_l=\sN_{l'}$. Since the stability functions coincide, the same simple objects of the heart are mapped to $-1$ and $i$, respectively. Now since the partial silting subcategory is determined by the simple objects sent to $-1$ via the silting subcategory versus simple object duality described in Section~\ref{sub:compatibility}, the subcategories $\sM'_l,\sN'_{l'}\subseteq\sM_l=\sN'_{l'}$ are determined by the stability function.
It is then clear from the definitions that $(\sM_l,\sM_l') = (\sN_{l'},\sN_{l'}')$ and so by construction, $(\sM_\fl,\sM'_\fl,a_\fl)$ and $ (\sN_\fl,\sN'_\fl,b_\fl)$ define the same point in the silting pairs CW complex.

Now suppose that any $(\sM_\fl,\sM'_\fl,a_\fl)$ and $(\sN_\fl,\sN'_\fl,b_\fl)$ with index $d \leq d'$, and satisfying $f(\sM_\fl,\sM'_\fl,a_\fl) = f(\sN_\fl,\sN'_\fl,b_\fl)$ correspond to the same point in the silting pairs CW complex. Take any $(\sM_\fl,\sM'_\fl,a_\fl)$ and $(\sN_\fl,\sN'_\fl,b_\fl)$ with index $d'+1$ which map to the same point under $f$. 
It follows from the definition that $\sH_{\sM_l} = \sH_{\sN_{l'}}$ and so $\sM_l = \sN_{l'}$.

Suppose that $\sM_l' \neq \sN_{l'}'$. Then there is some simple object $S \in \sH_{\sM_l} = \sH_{\sN_{l'}}$ such that $\charge_{\sM_l,\sM_l'}(S) =  -1$, but $\charge_{\sN_{l'},\sN_{l'}'}(S)=i$ (or vice-versa). It follows from Corollary~\ref{cor:real} that in fact $\charge_{\sM_i,\sM_i'}(S) \in \IR$ for all $i \geq l$. Looking at the imaginary part of $\sum_{j=0}^k a_j \charge_{\sM_j,\sM_j'} =\sum_{j=0}^k b_j \charge_{\sN_j,\sN_j'}$ evaluated on $S$, using the fact that $\charge_{\sM_j,\sM_j'}$ and $\charge_{\sN_j,\sN_j'}$ map the simple objects of $\sH_{\sM_l} = \sH_{\sN_{l'}}$ into the closed upper half-plane, we see that $b_{l'}=0$. However, this is a contradiction. Therefore, $\sM_l' = \sN_{l'}'$.

We use a similar argument to show that $a_l = b_{l'}$. Since the rank of $\sM_j'$ strictly decreases as $j$ increases, we can find a simple object $S$ in $\sH_{\sM_l} = \sH_{\sN_{l'}}$ such that $\charge_{\sM_l,\sM_l'}(S) = \charge_{\sN_{l'},\sN_{l'}'}(S) = i$, but $\charge_{\sM_{l+j},\sM_{l+j}'}(S) \in \IR$ for all $j >0$. Looking again at the imaginary part of the stability function evaluated at $S$, we see that \[ a_l = b_{l'} + \sum_{j=l'}^k b_j \text{Im}(\charge_{\sN_{j},\sN_{j}'}(S))\]
and since all the terms are positive it follows that $ a_l \geq b_{l'}$. The result follows using symmetry.

For any $t \in [0, 1]$ we define \[ \lambda_t \coloneqq \frac{1 - a_l(1-t)}{1-a_l} = \frac{1 - b_{l'}(1-t)}{1-b_{l'}}\]
(this is well defined since $d'+1>1$ implies $a_l \neq 1 \neq b_{l'}$) and consider the new pair of points $(\sM_\fl,\sM'_\fl,\widetilde a_\fl(t))$ and $(\sN_\fl,\sN'_\fl,\widetilde b_\fl(t))$ where

\[ \widetilde a_j(t) \coloneqq \begin{cases}
                                      (1-t)a_l & j=l, \\
                                      \lambda_t a_j             & \text{otherwise}
                                    \end{cases} \quad\text{and}\quad
   \widetilde b_j(t) \coloneqq \begin{cases}
                                      (1-t)b_{l'} & j=l', \\
                                      \lambda_t b_j             & \text{otherwise}
                                    \end{cases}
\]

We note that for $t=0$ we have the same pair of points as before, and for $t \in [0,1)$, the hearts $\sH_{\sM} = \sH_{\sN}$ of the corresponding stability functions are the same. Furthermore, since
\[ \sum_{j=l}^k a_j \charge_{\sM_j,\sM_j'} =\sum_{j=l'}^k b_j \charge_{\sN_j,\sN_j'} \] and we have shown that $\sM_l' = \sN_{l'}'$ and $a_l = b_{l'}$, it is clear that
\[ \lambda_t \sum_{j=l+1}^k a_j \charge_{\sM_j,\sM_j'} = \lambda_t \sum_{j=l'+1}^k b_j \charge_{\sN_j,\sN_j'}. \] Therefore  $f(\sM_\fl,\sM'_\fl,\widetilde a_\fl(t)) = f(\sN_\fl,\sN'_\fl, \widetilde b_\fl(t))$ for all $t \in [0,1)$. Continuity then gives the equality for all $t \in [0,1]$. However, by construction, $\widetilde a_l(1) = 0 = \widetilde b_{l'}(1)$ and so the pair of points $(\sM_\fl,\sM'_\fl,\widetilde a_\fl(1))$ and $(\sN_\fl,\sN'_\fl,\widetilde b_\fl(1))$ has degree $d\leq d'$. Invoking the induction hypothesis we conclude that the points are the same in the silting pairs CW complex. All of the points $(\sM_\fl,\sM'_\fl,\widetilde a_\fl(t))$ and $(\sN_\fl,\sN'_\fl,\widetilde b_\fl(t))$ lie in a $(d'+1)$-dimensional simplex which has $(\sM_l, \sM_l') =(\sN_{l'}, \sN_{l'}')$ as one of the generators. The parametrisation varies the coefficient of this generator in exactly the same way for both points. It is therefore clear that for each $t \in [0,1]$ (and in particular for $t=0$), the points $(\sM_\fl,\sM'_\fl,\widetilde a_\fl(t))$ and $(\sN_\fl,\sN'_\fl,\widetilde b_\fl(t))$ coincide in the silting pairs CW complex. This completes the induction.
\end{proof}

\begin{lemma} $f^{-1} \colon \image(f) \to \silt{\sD}$ is continuous.
\end{lemma}

We thank the referee for the following argument.

\begin{proof}
For any compact subset $K\subset\silt{\sD}$, the restriction $f|_K\colon K\to f(K)$ is a homeo\-morphism, because it is a continuous bijection from a compact space to a Hausdorff space. As $\silt{\sD}$ is locally compact, the map $f$ is continuous on all of $\silt{\sD}$.
\end{proof}

\subsection{Contractibility}

There is a folklore belief that when $\stab{\sD}$ is non-empty it is contractible. This is known in only a very few cases. For instance, in representation theory, it is known only for the bounded derived categories of
the algebras $\kk A_2, \kk\tilde A_1, \kk\tilde A_2, \Lambda(1,2,0)$; see \cite{BQS,Okada,DK,Koenig-Yang}.
As a corollary of our Theorem~\ref{thm:CW-contractible} and the work of Qiu and Woolf in \cite{QW}, we are able to add the family of non-hereditary, finite global dimension derived-discrete algebras to this list.

\begin{theorem} \label{thm:stab-contractible}
The stability manifold $\stab{\Db(\LLambda)}$ is contractible.
\end{theorem}

\begin{proof}
By Proposition~\ref{prop:length}, all the hearts in $\sD \coloneqq \Db(\LLambda)$ are module categories of finite-dimensional algebras. This means that all stability conditions on $\sD$ are algebraic in the sense of \cite{QW}. In particular, via the correspondences of K\"onig and Yang in \cite{Koenig-Yang} (see Theorem~\ref{thm:Koenig-Yang}) and the compatibility with mutation, the posets $\posetsilt(\sD)$ and $\posetpair(\sD)$ in Section~\ref{sec:silting-pairs} of this paper correspond to the posets $\mathrm{Tilt}_{\mathrm{alg}}(\sD)^{\mathrm{op}}$ and $\mathrm{Int}(\sD) \cup \{\hat{0}\}$ respectively in \cite{QW}, where $\hat{0}$ is a formally adjoined bottom element.

In \cite{QW}, it is shown that each component of $\stab{\sD}$ deformation-retracts onto a component of a CW complex induced from $\mathrm{Int}(\sD)$. They use this to deduce that each component of $\stab{\sD}$ is contractible. We mention that \cite{QW} uses topological results from \cite{FMT}; the topological backbone for Section~\ref{sec:CW-poset} was \cite{BjoernerEJC}.

We can, however, say more: using Theorem~\ref{thm:CW-contractible}, we know that the CW complex induced from the poset $\posetpair(\sD)$, and thus from $\mathrm{Int}(\sD)$, is contractible. Therefore, it follows that the whole space $\stab{\sD}$ is contractible.
\end{proof}


\appendix

\section{The K\"onig--Yang correspondences} \label{app:examples}

\noindent
In this appendix we explicitly describe the K\"onig--Yang correspondences of Theorem~\ref{thm:koenig-yang} when $\Lambda$ has finite global dimension and give some concrete examples in Dynkin type $A_3$. 
The reference is, of course, \cite{Koenig-Yang}, and the references therein.

\subsection{The correspondences}

Let $\Lambda$ be a finite-dimensional algebra of finite global dimension. Let $\sM \subseteq \Db(\Lambda) \simeq \Kb(\proj{\Lambda})$ be a silting subcategory. Then $\sM$ determines an algebraic t-structure $(\sX_\sM, \sY_\sM)$ as follows:
\[
\sX_\sM \coloneqq (\Sigma^{<0} \sM)\orth \quad \text{and} \quad  \sY_\sM \coloneqq (\Sigma^{\geq 0} \sM)\orth,
\]
and a bounded co-t-structure
\[
\sA_\sM \coloneqq \cosusp{\Sigma^{-1} \sM} = \bigcup_{l > 0} \Sigma^{-l} \sM * \cdots * \Sigma^{-1} \sM
\text{ and }
\sB_\sM \coloneqq \susp{\sM} = \bigcup_{l \geq 0} \sM * \cdots * \Sigma^l \sM.
\]

Now given an algebraic t-structure $(\sX, \sY)$, we obtain a bounded co-t-structure and a silting subcategory as follows:
\[
(\sA, \sB) \coloneqq ({}\orth \sX, \sX) 
\text{ and }
\sM \coloneqq \Sigma ({}\orth \sX) \cap \sX.
\]
In the terminology of \cite{Bondarko} we say that the co-t-structure $(\sA,\sB)$ is \emph{left adjacent} to the t-structure $(\sX,\sY)$. 

Finally, given a bounded co-t-structure $(\sA,\sB)$ we define a silting subcategory $\sM$ and an algebraic t-structure as follows:
\[
\sM \coloneqq \Sigma \sA \cap \sB
\text{ and }
(\sX, \sY) \coloneqq (\sB, \sB\orth).
\]
The silting subcategory $\sM$ is the \emph{co-heart} of $(\sA,\sB)$ \cite{Pauksztello} and the t-structure is \emph{right adjacent} to the co-t-structure.

\subsection{Examples}

In each of the diagrams below we depict the indecomposable objects of the AR quiver of $\Db(\kk A_3)$. 
Each diagram depicts a bounded t-structure together with the corresponding bounded co-t-structure, silting subcategory and the simple objects of the heart. We highlight these subcategories and collections of objects as follows.

\noindent
\begin{center}
\begin{tabular}{@{} p{0.15\textwidth} @{} p{0.73\textwidth} @{}}
symbols                                                     & indecomposable objects of \\ \midrule
\tikz{\draw (0,0) circle (1mm);}                             & the aisle of the t-structure (also co-aisle of  the co-t-structure) \\
\tikz{\draw (0.1, 0.1) rectangle (-0.1,-0.1);}               & the co-aisle of the t-structure \\
\tikz{\draw (7, 2 + 0.1) -- (7 + 0.1, 2 - 0.1) -- (7 - 0.1, 2 - 0.1) -- cycle;} & neither aisle nor co-aisle of the t-structure \\
\tikz{\draw[fill=gray!80] (0,0) circle (1mm);} or 
\tikz{\draw[fill=gray!80,very thick] (0,0) circle (1mm);}    & the silting subcategory (also co-heart of co-t-structure) \\
\tikz{\draw[very thick] (0,0) circle (1mm);} or 
\tikz{\draw[fill=gray!80,very thick] (0,0) circle (1mm);}    & simple objects of the heart of the t-structure \\
\tikz{\draw[fill=gray!30] (7, 2 + 0.1) -- (7 + 0.1, 2 - 0.1) -- (7 - 0.1, 2 - 0.1) -- cycle;} or
\tikz{\draw[fill=gray!30] (0.1, 0.1) rectangle (-0.1,-0.1);} & the aisle of the co-t-structure \\
\multicolumn{2}{@{} l}{Hearts of the t-structures are indicated by triangle outlined regions:
                    \tikz{\draw[rounded corners] (10, 2.5) -- (10.5, 1.75) -- (9.5, 1.75) -- cycle;}. }
\end{tabular}
\end{center}

\bigskip
\noindent
Below we consider the standard t-structure. Using the orientation $1\to2\to3$, the silting object is $P(1)\oplus P(2)\oplus P(3)$; it is the standard tilting module for $\mod{\kk A_3}$ with this orientation.

\begin{center}
\begin{tikzpicture}

\foreach \x in {0,1,...,5} \foreach \y in {0}
  { \fill[gray!30] (\x - 0.1, \y + 0.1) -- (\x + 0.1, \y + 0.1) -- (\x + 0.1, \y - 0.1) -- (\x - 0.1, \y - 0.1) -- cycle; }

\foreach \x in {0.5,1.5,...,4.5} \foreach \y in {1}
  { \fill[gray!30] (\x - 0.1, \y + 0.1) -- (\x + 0.1, \y + 0.1) -- (\x + 0.1, \y - 0.1) -- (\x - 0.1, \y - 0.1) -- cycle; }

\foreach \x in {0,1,...,4} \foreach \y in {2}
  { \fill[gray!30] (\x - 0.1, \y + 0.1) -- (\x + 0.1, \y + 0.1) -- (\x + 0.1, \y - 0.1) -- (\x - 0.1, \y - 0.1) -- cycle; }

\fill[gray!80] (6,0)   circle (1mm);
\fill[gray!80] (6.5,1) circle (1mm);
\fill[gray!80] (7,2)   circle (1mm);

\foreach \x in {0,1,...,5} \foreach \y in {0}
  { \draw (\x - 0.1, \y + 0.1) -- (\x + 0.1, \y + 0.1) -- (\x + 0.1, \y - 0.1) -- (\x - 0.1, \y - 0.1) -- cycle; }

\foreach \x in {0.5,1.5,...,5.5} \foreach \y in {1}
  { \draw (\x - 0.1, \y + 0.1) -- (\x + 0.1, \y + 0.1) -- (\x + 0.1, \y - 0.1) -- (\x - 0.1, \y - 0.1) -- cycle; }

\foreach \x in {0,1,...,6} \foreach \y in {2}
  { \draw (\x - 0.1, \y + 0.1) -- (\x + 0.1, \y + 0.1) -- (\x + 0.1, \y - 0.1) -- (\x - 0.1, \y - 0.1) -- cycle;}

\foreach \x in {6,7,...,14} \foreach \y in {0} { \draw (\x, \y) circle (1mm); }

\foreach \x in {6.5,7.5,...,13.5} \foreach \y in {1} { \draw (\x, \y) circle (1mm); }

\foreach \x in {7,8,...,14} \foreach \y in {2} { \draw (\x, \y) circle (1mm); }

\draw[rounded corners] (7,2.5) -- (8.5,-0.25) -- (5.5,-0.25) -- cycle; 

\draw[very thick] (6,0) circle (1mm);
\draw[very thick] (7,0) circle (1mm);
\draw[very thick] (8,0) circle (1mm);
\end{tikzpicture}
\end{center}

\bigskip
\noindent
Here we consider an irreducible right mutation of the standard t-structure: we take the subobject
 $M' = P(1)\oplus P(2) \subset M = P(1)\oplus P(2)\oplus P(3)$ and arrive at the new silting object 
$\rmu{M'}(M) = P(1) \oplus P(2) \oplus \Sigma\inv P(3)$. This is a non-tilting silting object, and hence does not live in the heart of the t-structure.
\begin{center}
\begin{tikzpicture}

\foreach \x in {0,1,...,2} \foreach \y in {0}
  { \fill[gray!30] (\x - 0.1, \y + 0.1) -- (\x + 0.1, \y + 0.1) -- (\x + 0.1, \y - 0.1) -- (\x - 0.1, \y - 0.1) -- cycle; }

\foreach \x in {0.5,1.5,...,2.5} \foreach \y in {1}
  { \fill[gray!30] (\x - 0.1, \y + 0.1) -- (\x + 0.1, \y + 0.1) -- (\x + 0.1, \y - 0.1) -- (\x - 0.1, \y - 0.1) -- cycle; }

\foreach \x in {0,1,...,3} \foreach \y in {2}
  { \fill[gray!30] (\x - 0.1, \y + 0.1) -- (\x + 0.1, \y + 0.1) -- (\x + 0.1, \y - 0.1) -- (\x - 0.1, \y - 0.1) -- cycle; }

\fill[gray!30] (5, 0 + 0.1) -- (5 + 0.1, 0 - 0.1) -- (5 - 0.1, 0 - 0.1) -- cycle;
\fill[gray!30] (4.5, 1 + 0.1) -- (4.5 + 0.1, 1 - 0.1) -- (4.5 - 0.1, 1 - 0.1) -- cycle;
\fill[gray!30] (4 - 0.1, 0 + 0.1) -- (4 + 0.1, 0 + 0.1) -- (4 + 0.1, 0 - 0.1) -- (4 - 0.1, 0 - 0.1) -- cycle;

\fill[gray!80] (7,2)   circle (1mm);
\fill[gray!80] (6.5,1) circle (1mm);
\fill[gray!80] (4,2)   circle (1mm);

\foreach \x in {0,1,...,4} \foreach \y in {0}
  { \draw (\x - 0.1, \y + 0.1) -- (\x + 0.1, \y + 0.1) -- (\x + 0.1, \y - 0.1) -- (\x - 0.1, \y - 0.1) -- cycle; }

\foreach \x in {0.5,1.5,...,3.5} \foreach \y in {1}
  { \draw (\x - 0.1, \y + 0.1) -- (\x + 0.1, \y + 0.1) -- (\x + 0.1, \y - 0.1) -- (\x - 0.1, \y - 0.1) -- cycle; }

\foreach \x in {0,1,...,3} \foreach \y in {2}
  { \draw (\x - 0.1, \y + 0.1) -- (\x + 0.1, \y + 0.1) -- (\x + 0.1, \y - 0.1) -- (\x - 0.1, \y - 0.1) -- cycle; }

\draw (6 - 0.1, 2 + 0.1) -- (6 + 0.1, 2 + 0.1) -- (6 + 0.1, 2 - 0.1) -- (6 - 0.1, 2 - 0.1) -- cycle;
\draw (5 - 0.1, 2 + 0.1) -- (5 + 0.1, 2 + 0.1) -- (5 + 0.1, 2 - 0.1) -- (5 - 0.1, 2 - 0.1) -- cycle;
\draw (5.5 - 0.1, 1 + 0.1) -- (5.5 + 0.1, 1 + 0.1) -- (5.5 + 0.1, 1 - 0.1) -- (5.5 - 0.1, 1 - 0.1) -- cycle;

\foreach \x in {6,7,...,14} \foreach \y in {0}
 { \draw (\x, \y) circle (1mm); }

\foreach \x in {6.5,7.5,...,13.5} \foreach \y in {1}
  { \draw (\x, \y) circle (1mm); }

\foreach \x in {7,8,...,14} \foreach \y in {2}
  { \draw (\x, \y) circle (1mm); }

\draw (4,2) circle (1mm);

\draw (5, 0 + 0.1) -- (5 + 0.1, 0 - 0.1) -- (5 - 0.1, 0 - 0.1) -- cycle;
\draw (4.5, 1 + 0.1) -- (4.5 + 0.1, 1 - 0.1) -- (4.5 - 0.1, 1 - 0.1) -- cycle; 
 
\draw[rounded corners] (7.5,1.5) -- (8.5,-0.25) -- (6.5,-0.25) -- cycle; 
\draw[rounded corners] (4, 2.5) -- (4.5, 1.75) -- (3.5, 1.75) -- cycle;

\draw[very thick] (8,0)  circle (1mm);
\draw[very thick] (7,0)  circle (1mm);
\draw[very thick] (4,2) circle (1mm);

\end{tikzpicture}
\end{center}

\FloatBarrier

\addtocontents{toc}{\protect{\setcounter{tocdepth}{-1}}}

\end{document}